\documentclass[10pt,reqno]{amsart}
\usepackage{graphicx}
\usepackage{amsfonts}
\usepackage{amssymb}
\usepackage{amsmath}
\usepackage{amsxtra}
\usepackage{latexsym}
\usepackage{epstopdf}
\usepackage{mathrsfs}
\usepackage{esint}

\newtheorem{theorem}{Theorem}[section]
\newtheorem{lemma}[theorem]{Lemma}
\newtheorem{corollary}{Corollary}[section]

\theoremstyle{definition}

\newtheorem{algorithm}{Algorithm}[section]
\newtheorem{example}{Example}[section]

\newtheorem{Assumption}{Assumption}[section]

\theoremstyle{remark}

\numberwithin{equation}{section}

\begin{document}

\title[Landweber-Kaczmarz method in Banach spaces]
{Landweber-Kaczmarz method in Banach spaces with inexact inner solvers}

\author{Qinian Jin}

\address{Mathematical Sciences Institute, Australian National
University, Canberra, ACT 2601, Australia}
\email{qinian.jin@anu.edu.au} \curraddr{}


\def\ep{\varepsilon}


\begin{abstract}
In recent years Landweber(-Kaczmarz) method has been proposed for solving nonlinear
ill-posed inverse problems in Banach spaces using general convex penalty functions.
The implementation of this method involves solving a (nonsmooth) convex minimization
problem at each iteration step and the existing theory requires its exact resolution
which in general is impossible in practical applications. In this paper we propose a
version of Landweber-Kaczmarz method in Banach spaces in which the minimization problem
involved in each iteration step is solved inexactly. Based on the $\ep$-subdifferential
calculus we give a convergence analysis of our method. Furthermore, using Nesterov's strategy,
we propose a possible accelerated version of Landweber-Kaczmarz method. Numerical
results on computed tomography and parameter identification in partial differential equations
are provided to support our theoretical results and to demonstrate our accelerated
method.
\end{abstract}

\def\p{\partial}
\def\d{\delta}
\def\l{\langle}
\def\r{\rangle}
\def\C{\mathcal C}
\def\D{\mathscr D}
\def\a{\alpha}
\def\b{\beta}
\def\d{\delta}

\def\la{\lambda}
\def\ep{\varepsilon}
\def\Ga{\Gamma}
\def\R{{\mathcal R}}
\def\J{{\mathcal J}}
\def\A{\mathcal A}
\def\B{\mathcal B}
\def\a{\alpha}
\def\d{\delta}

\def\bx{\bf x}
\def\by{{\bf y}}
\def\bA{\bf A}
\def\bV{\bf V}
\def\bW{\bf W}

\def\N{\mathcal N}
\def\R{\mathcal R}
\def\X{\mathcal X}
\def\Y{\mathcal Y}
\def\B{\mathcal B}
\def\A{\mathcal A}
\def\H{\mathcal H}
\def\bA{{\bf A}}
\def\bx{{\bf x}}
\def\bb{{\bf b}}
\def\ba{{\bf a}}
\def\bV{{\bf V}}
\def\bW{{\bf W}}
\def\bQ{{\bf Q}}
\def\bD{{\bf D}}
\def\D{\mathscr D}

\maketitle
\section{\bf Introduction}

Regularization of inverse problems has been considered extensively and significant progress has been made; 
see \cite{EHN96,IJ2014,JT2009,KNS2008,TJ2003} and references therein. Due to the demand of capturing special
features of the reconstructed objects and the demand of dealing with general noise, regularization in Banach
spaces has emerged as a highly active research field and many new regularization methods have been proposed and investigated
in recent years; one may refer to \cite{BH2012,HW2013,Jin2015,JZ2013,JZ2014,KH2010,SKHK2012} and the references therein for recent developments.

Because of its simplicity and relatively small complexity per iteration, Landweber iteration and its Kaczmarz version
have received extensive attention in inverse problem community \cite{CHLS2008,HLS2007,HKLS2007,HNS95,KS2002}. 
In recent years, several versions of Landweber iteration
has been formulated in Banach spaces, see \cite{BH2012,Jin2012,JW2013,SLS2006}. When solving ill-posed system of the
form
\begin{equation}\label{sys0}
F_i(x) = y_i, \qquad i=0, \cdots, N-1
\end{equation}
consisting of $N$ equations, a Kaczmarz version of Landweber iteration in Banach spaces with general
uniformly convex penalty functions has been proposed in \cite{JW2013} which cyclically considers each equation
in (\ref{sys0}) in a Gauss-Seidel manner.  For these modern versions of Landweber method, each iteration step
essentially requires the computation of next iterate $\xi_+, x_+$ from current iterate $\xi_c, x_c$ via
\begin{align*}
\left\{\begin{array}{lll}
\xi_+ = \xi_c - t F'(x_c)^* J (F(x_c)-y),\\[0.8ex]
x_+ = \arg \displaystyle{\min_x\left\{ \Theta(x) - \l \xi_+, x\r \right\}},
\end{array}\right.
\end{align*}
where $t>0$ is a step size, $J$ is a duality mapping, $F, y$ denote one of $F_i, y_i$, $F'(x)$ denotes the Fr\'{e}chet derivative
of $F$, and $\Theta$ is a uniformly convex function. Therefore, the implementation of the Landweber(-Kaczmarz) method
in Banach spaces requires to solving a minimization problem of the form
\begin{equation}\label{min0}
x = \arg \displaystyle{\min_z\left\{ \Theta(z) - \l \xi, z\r \right\}}
\end{equation}
associated with $\Theta$ in each iteration step.

The existing convergence theory on Landweber(-Kaczmarz) method in Banach spaces requires the exact resolution
of the minimization problem (\ref{min0}). For some special $\Theta$ its exact resolution is possible.
However, this minimization problem in general can only be solved inexactly by an iterative procedure which
may produce an approximate solution $\bar x$ satisfying
\begin{equation}\label{min1}
\Theta(\bar x) - \l \xi, \bar x\r \le \arg \displaystyle{\min_z\left\{ \Theta(z) - \l \xi, z\r \right\}} +\ep
\end{equation}
for some small $\ep>0$. Furthermore, numerical simulations indicate that solving (\ref{min0}) very accurately
in every step does not improve the final reconstruction result but wastes a huge amount of computational time. Therefore,
it is necessary to formulate a Landweber-Kaczmarz method with inexact inner solver in each iteration step
and to develop the corresponding convergence theory. The inequality (\ref{min1}) suggests that the $\ep$-subdifferential
calculus might be a useful tool for this purpose.

It is well-known that Landweber iteration admits the slow convergence property (\cite{EHN96}) which restricts
its applications to wide range of problems. To expand its applied range, it is necessary to introduce
some acceleration strategy to fasten its convergence speed.  In this paper we will use the Nesterov's
strategy in optimization (\cite{Nest1983}) to propose an accelerated version of Landweber-Kaczmarz method in Banach spaces
in which some extrapolation steps are incorporated. We do not have a theory to guarantee its acceleration effect
at this moment, however, we will provide numerical simulations to support the fast convergence property.

This paper is organized as follows. In section 2 we will provide some preliminaries
on Banach spaces and convex analysis and derive some useful results concerning
$\ep$-subdifferential. In section 3 we propose the Landweber-Kaczmarz method
with inexact inner solvers, show that it is well-defined, and prove its convergence and
regularization property. We then use Nesterov's strategy to propose an accelerated version.
We also discuss how to produce the inexact solvers for solving the inner minimization problem
at each iteration step of Landweber-Kaczmarz method. Finally, we provide numerical simulations
to verify the theoretical results and to demonstrate the fast convergence of the accelerated
method.

\section{\bf Preliminaries}

Let $\X$ be two Banach space whose norm is denoted by $\|\cdot\|$. We use $\X^*$ 
to denote its dual spaces. For any $x\in \X$ and $\xi\in \X^*$, we write $\l \xi, x\r=\xi(x)$
for the duality pairing. Let $\Y$ be another Banach space. By ${\mathscr L}(\X, \Y)$ we denote for 
the space of all bounded linear operators from $\X$ to $\Y$.  For any $A \in {\mathscr L}(\X, \Y)$ we use
$A^*: \Y^*\to \X^*$ to denote its adjoint, i.e.
\begin{equation*}
\l A^*\zeta, x\r = \l \zeta, A x\r
\end{equation*}
for any $x\in \X$ and $\zeta \in \Y^*$. 

For each $1<s<\infty$, the set-valued mapping $J_s^{\X}: \X\to 2^{\X^*}$ defined by
\begin{equation*}
J_s^\X(x):=\left\{\xi\in \X^*: \|\xi\|=\|x\|^{s-1} \mbox{ and } \l \xi, x\r=\|x\|^s\right\}
\end{equation*}
is called the duality mapping of $\X$ with gauge function $t\to t^{s-1}$. When $\X$ is uniformly
smooth in the sense that its modulus of smoothness
\begin{equation*}
\rho_{\X}(t) := \sup\{\|\bar x+ x\|+\|\bar x-x\|- 2 : \|\bar x\| = 1,  \|x\|\le t\}
\end{equation*}
satisfies $\lim_{t\searrow 0} \frac{\rho_{\X}(t)}{t} =0$, the duality mapping $J_s^\X$, for each $1<s<\infty$,
is single valued and uniformly continuous on bounded sets.

Given a convex function $\Theta: \X \to (-\infty, \infty]$, we use
\begin{equation*}
\D(\Theta): =\{x\in \X: \Theta(x)<\infty\}
\end{equation*}
to denote its effective domain. It is called proper if $\D(\Theta)\ne \emptyset$. For a proper convex function
$\Theta: \X \to (-\infty, \infty]$ and $x\in \X$, we define for any $\ep\ge 0$ the set
\begin{equation*}
\p_\ep \Theta(x) :=\{ \xi \in \X^*: \Theta(\bar x) \ge \Theta(x) +\l \xi, \bar x- x\r -\ep \mbox{ for all } \bar x\in \X\}
\end{equation*}
which is called the $\ep$-subdifferential of $\Theta$ at $x$. Any element in $\p_\ep \Theta(x)$ is called
an $\ep$-subgradient of $\Theta$ at $x$. When $\ep=0$, the $\ep$-subdifferential of $\Theta$ reduces to
the subdifferential $\p \Theta$. It is clear that $\p_\ep \Theta(x) \ne \emptyset$ for some
$\ep\ge 0$ implies $x\in \D(\Theta)$. If $\Theta$ is lower semi-continuous, then for any $x\in \D(\Theta)$,
the $\ep$-subdifferential $\p_\ep \Theta(x)$ is always non-empty for any $\ep>0$, see \cite[Theorem 2.4.4]{Z2002};
however, $\p \Theta(x)$ can be empty in general.

For $\xi\in \p_\ep \Theta(x)$ with $\ep\ge 0$, we may introduce
\begin{equation*}
D_{\xi}^\ep \Theta (\bar x, x) := \Theta(\bar x) -\Theta(x) -\l \xi, \bar x-x\r +\ep, \quad \forall \bar x\in \X
\end{equation*}
which is called the $\ep$-Bregman distance induced by $\Theta$ at $x$ in the direction $\xi$. It is clear that
\begin{equation*}
D_\xi^\ep \Theta(\bar x, x) \ge 0.
\end{equation*}
When $\ep=0$, the $\ep$-Bregman distance becomes the well-known Bregman
distance \cite{Br1967} which will be denoted by $D_\xi \Theta(\bar x, x)$. It should be pointed out that $\ep$-Bregman
distance is not a metric distance in general. Nevertheless, as the following result shows,
the $\ep$-Bregman distance can be used to detect information under the norm if $\Theta$ is $p$-convex for
some $p\ge 2$ in the sense that  there is a constant $c_0>0$ such that
\begin{equation}\label{p-conv}
\Theta(t \bar x +(1-t) x)+c_0 t(1-t) \|\bar x-x\|^p \le t \Theta(\bar x) +(1-t) \Theta(x)
\end{equation}
for all $0\le t\le 1$ and $\bar x, x\in \X$, .

\begin{lemma}\label{lem2}
Let $\Theta: \X \to (-\infty, \infty]$ be a proper, lower semi-continuous function that is $p$-convex in the
sense of (\ref{p-conv}). If $\xi\in \p_\ep \Theta(x)$ for some $\ep\ge 0$, then
\begin{equation}\label{1.15}
c_0 \|\bar x-x\|^p \le 2 D_\xi^\ep \Theta(\bar x, x) + 2\ep.
\end{equation}
for any $\bar x \in \X$.
\end{lemma}

\begin{proof}
Since $\Theta$ is $p$-convex and $\xi\in \p_\ep \Theta(x)$, we have for any $0<t<1$ that
\begin{align*}
c_0 t (1-t) \|\bar x-x\|^p &\le t \Theta(x) +(1-t) \Theta(\bar x) - \Theta(t x +(1-t) \bar x)\\
& \le t \Theta(x) +(1-t) \Theta(\bar x) -\left[\Theta(x) +(1-t) \l \xi, \bar x-x\r -\ep\right]\\
& = (1-t) \left[ \Theta(\bar x)-\Theta(x) -\l \xi, \bar x-x\r +\ep\right]+ t\ep \\
& = (1-t) D_\xi^\ep \Theta(\bar x, x) +t\ep.
\end{align*}
By taking $t =1/2$ we then obtain (\ref{1.15}).
\end{proof}

In convex analysis, the Legendre-Fenchel conjugate is an important notion. Given a proper convex function
$\Theta: \X \to (-\infty, \infty]$, its Legendre-Fenchel conjugate is defined by
\begin{equation*}
\Theta^*(\xi) := \sup_{x\in \X} \left\{ \l \xi, x\r - \Theta(x)\right\}, \qquad \forall \xi\in \X^*.
\end{equation*}
As an immediate consequence of the definition, one can see, for any $\ep\ge 0$, that
\begin{equation}\label{4.10.2014}
\xi\in \p_\ep \Theta(x) \Longleftrightarrow \Theta(x) + \Theta^*(\xi) \le \l \xi, x\r +\ep.
\end{equation}
If, in addition,  $\Theta$ is lower semi-continuous, then (\cite[Theorem 2.4.4]{Z2002})
\begin{equation}\label{4.11.2014}
\xi\in \p_\ep \Theta(x) \Longleftrightarrow x\in \p_\ep \Theta^*(\xi).
\end{equation}
For a proper, lower semi-continuous, $p$-convex function, even if it is non-smooth, its Legendre-Fenchel conjugate
can have enough regularity as the following result indicates.

\begin{lemma}\label{Fenchel}
Let $\X$ be a reflexive Banach space and let $\Theta:\X \to (-\infty, \infty]$ be a proper, lower semi-continuous 
function that is $p$-convex in the sense
of (\ref{p-conv}). Then $\D(\Theta^*)=\X^*$, $\Theta^*$ is Fr\'{e}chet differentiable, and its gradient
$\nabla \Theta^*: \X^*\to \X$ satisfies
\begin{equation}\label{10.13.4}
\|\nabla \Theta^*(\xi)-\nabla \Theta^*(\eta)\|\le \left(\frac{\|\xi-\eta\|}{2c_0}\right)^{\frac{1}{p-1}}
\end{equation}
which consequently implies 
\begin{equation}\label{Lip}
|\Theta^*(\eta)-\Theta^*(\xi)-\l\eta-\xi, \nabla \Theta^*(\xi)\r |\le \frac{1}{p^* (2c_0)^{p^*-1}} \|\xi-\eta\|^{p^*}
\end{equation}
for any $\xi, \eta\in \X^*$, where $p^*$ is the number conjugate to $p$, i.e. $1/p+1/p^*=1$.
\end{lemma}

\begin{proof}
See \cite[Theorem 3.5.10 and Corollary 3.5.11]{Z2002}.
\end{proof}

Finally we conclude this section by providing a result which show that, when $\Theta$ is $p$-convex,
then, for any $x\in \p_\ep \Theta^*(\xi)$, the distance from $x$ to $\nabla \Theta^*(\xi)$ can be
controlled in terms of $\ep$.

\begin{lemma}\label{lem3}
Let $\X$ ba a reflexive Banach space and let $\Theta: \X \to (-\infty, \infty]$ be a proper, lower semi-continuous 
function that is $p$-convex in
the sense of (\ref{p-conv}). If $x\in \X$ and $\xi\in \X^*$ satisfy $\xi\in \p_\ep \Theta(x)$ for some $\ep\ge 0$, 
then for any $\eta\in \X^*$ there holds
\begin{equation}\label{eq37}
\l \eta, x-\nabla \Theta^*(\xi)\r \le \ep + \frac{1}{p^*(2c_0)^{p^*-1}} \|\eta\|^{p^*}
\end{equation}
and hence 
\begin{equation}\label{eq38}
\|x-\nabla \Theta^*(\xi)\|^p\le \frac{p}{2c_0}\ep.
\end{equation}
\end{lemma}

\begin{proof}
Since $\xi\in \p_\ep \Theta(x)$, by (\ref{4.11.2014}) we have $x\in \p_\ep \Theta^*(\xi)$ and hence
\begin{equation*}
\Theta^*(\xi+\eta) \ge \Theta^*(\xi) + \l \eta, x\r -\ep.
\end{equation*}
By using (\ref{Lip}) in Lemma \ref{Fenchel}, we also have
\begin{equation*}
\Theta^*(\xi+\eta) \le \Theta^*(\xi) + \l \eta, \nabla \Theta^*(\xi)\r + \frac{1}{p^*(2c_0)^{p^*-1}} \|\eta\|^{p^*}.
\end{equation*}
Combining the above two inequalities we obtain
\begin{equation*}
 \l \eta, x\r -\ep \le \l \eta, \nabla \Theta^*(\xi)\r + \frac{1}{p^*(2c_0)^{p^*-1}} \|\eta\|^{p^*}
\end{equation*}
which shows (\ref{eq37}). By taking $\eta \in 2 c_0 J_p^{\X} (x-\nabla \Theta^*(\xi))$ in (\ref{eq37})
and using the properties of $J_p^{\X}$, we then obtain
\begin{equation*}
2 c_0 \|x-\nabla \Theta^*(\xi)\|^p \le \ep + \frac{2 c_0}{p^*} \|x-\nabla \Theta^*(\xi)\|^p
\end{equation*}
which shows (\ref{eq38}).
\end{proof}

\section{\bf The method}

We consider the system
\begin{equation}\label{sys}
F_i(x) = y_i, \qquad i=0, \cdots, N-1
\end{equation}
consisting of $N$ equations,  where, for each $i=0, \cdots, N-1$, $F_i: \D(F_i) \subset \X\to \Y_i$
is an operator between two reflexive Banach spaces $\X$ and $\Y_i$. Such systems arise in many practical applications
including various tomography problems using multiple exterior measurements.

We will assume that (\ref{sys}) has a solution which consequently implies that
\begin{equation*}
\D : = \bigcap_{i=0}^{N-1} \D(F_i) \ne \emptyset.
\end{equation*}
In practical applications, instead of $y_i$ we only have noisy data $y_i^\d$ satisfying
\begin{equation}\label{8.10.2}
\|y_i^\d -y_i\| \le \d, \qquad i=0, \cdots, N-1
\end{equation}
with a small noise level $\d>0$. How to use $y_i^\d$ to produce an approximate solution of (\ref{sys})
is an important question. In \cite{JW2013} we proposed a Landweber iteration of Kaczmarz type which makes
use of every equation in (\ref{sys}) cyclically.  In order to capture the feature of
the sought solution, general convex functions $\Theta: \X \to (-\infty, \infty]$ have been used in
\cite{JW2013} as penalty terms.

We will make the following assumption, where $B_\rho(x_0) :=\{x\in \X: \|x-x_0\|\le \rho\}$.

\begin{Assumption}\label{A1}
{\it 
\begin{enumerate}
\item[(a)] $\Theta: \X\to (-\infty, \infty]$ is proper, lower semi-continuous and $p$-convex
in the sense of (\ref{p-conv}).

\item[(b)] There exist $\rho > 0$, $x_0\in \X$ and $\xi_0\in \p \Theta(x_0)$ such that
$B_{2\rho}(x_0)\subset \D$ and (\ref{sys}) has a solution $x^\dag \in \D(\Theta)$ with
\begin{equation}\label{12.5}
D_{\xi_0} \Theta(x^\dag, x_0) \le \frac{1}{4} c_0 \rho^p;
\end{equation}

\item[(c)] For each $i=0, \cdots, N-1$ there exists $\{L_i(x): \X\to \Y_i\}_{x\in B_{2\rho}(x_0)} \subset {\mathscr L}(\X, \Y_i)$
such that $x\to L_i(x)$ is continuous on $B_{2\rho}(x_0)$ and there is $0\le \gamma<1$ such that
\begin{equation*}
\|F_i(\bar x)-F_i(x) -L_i(x) (\bar x-x)\|\le \gamma \|F_i(\bar x) -F_i(x)\|
\end{equation*}
for all $\bar x, x \in B_{2\rho}(x_0)$.
\end{enumerate}
}
\end{Assumption}

According to Assumption \ref{A1} (c), we can find a constant $B>0$ such that
\begin{equation}\label{eq:L}
\|L_i(x) \|\le B \qquad \forall x\in B_{2\rho}(x_0) \mbox{ and } i=0, \cdots, N-1.
\end{equation}
Moreover
\begin{equation}\label{12.6.1}
\|F_i(\bar x) -F_i(x) \| \le \frac{1}{1-\gamma} \|L_i(x) (\bar x-x)\| \le \frac{B}{1-\gamma}\|\bar x-x\|
\end{equation}
for all $\bar x, x \in B_{2\rho}(x_0)$ which shows that $F_i$ is continuous on $B_{2\rho}(x_0)$ for
each $i=0, \cdots, N-1$.

The formulation of the Landweber iteration of Kaczmarz type in \cite{JW2013} involves in each
iteration step a minimization problem of the form
\begin{equation}\label{min}
x: =\arg \min_{z\in \X} \left\{ \Theta(z) - \l \xi, z\r\right\}
\end{equation}
for any $\xi\in \X^*$. The convergence result developed there requires to solving (\ref{min}) exactly.
The exact solution of (\ref{min}) can be found for some special $\Theta$. However, this minimization
problem in general can only be solved inexactly by iterative procedures. Furthermore, numerical simulations
indicate that solving (\ref{min}) very accurately in every step does not improve the final reconstruction result.
Therefore, it is necessary to formulate a Landweber-Kaczmarz method with inexact inner solver
in each iteration step and to develop the corresponding convergence theory.

Concerning the inexact resolution of (\ref{min}), we make the following assumption.

\begin{Assumption}\label{A2}
{\it For any given $\ep>0$ there is a procedure $S_\ep:\X^*\to \X$ for solving (\ref{min}) such that for any $\xi\in \X^*$, the element
$x := S_\ep (\xi)$ satisfies
\begin{equation}\label{min_inexact}
\Theta(x) -\l\xi, x\r \le \min_{z\in \X} \left\{\Theta(z)-\l \xi, z\r\right\}+\ep.
\end{equation}
Moreover, for each $\ep>0$, the mapping $S_\ep: \X^*\to \X$ is continuous.
}
\end{Assumption}

In subsection 3.5 we will discuss how to produce the inexact procedure $S_\ep$ by using concrete
examples of $\Theta$ including the total variation like convex penalty functions.

\subsection{The method with noisy data}

We are ready to formulate our Landweber-Kaczmarz method with inexact inner solvers.
We will take $1<s<\infty$ and let $J_s^{\Y_i}$ denote the duality mapping over $\Y_i$ with gauge
function $t\to t^{s-1}$. Given an integer $n$, we set $i_n := n \, (\mbox{mod } N)$.

\begin{algorithm}[Landweber-Kaczmarz method with noisy data] \label{alg1} \quad

\noindent
{\it Let $\beta_0>0$, $\beta_1>0$ and $\tau>1$ be suitably chosen numbers,
and let $\{\ep_n\}_{n\ge 0}$ be a sequence of positive numbers satisfying $\sum_{n=0}^\infty \ep_n <\infty$.

\begin{enumerate}

\item[(i)] Pick $x_0\in \X$ and $\xi_0\in \X^*$ such that $\xi_0 \in \p \Theta(x_0)$. 

\item[(ii)] Let $\xi_0^\d:=\xi_0$ and $x_0^\d:=x_0$. Let $q_{-1}=0$ and let $\sigma>0$ be a small number.  
For $n\ge 0$ we define
$
r_n^\d = F_{i_n}(x_n^\d) -y_{i_n}^\d
$
and 
\begin{equation*}
q_n  = \left\{\begin{array}{lll}
q_{n-1}+1 \ & \mbox{ if } \|r_n^\d\|^p + \sigma \ep_n \le (\tau\d)^p,\\[0.8ex]
0 & \mbox{ otherwise.}
\end{array}\right.
\end{equation*}
We then update 
\begin{equation}\label{6.27.0}
\left\{\begin{array}{lll}
\xi_{n+1}^\d  =\xi_n^\d - \mu_n^\d L_{i_n}(x_n^\d)^* J_s^{\Y_{i_n}} (r_n^\d),\\[1.2ex]
x_{n+1}^\d  = S_{\ep_{n+1}}(\xi_{n+1}^\d),
\end{array}\right.
\end{equation}
where
\begin{equation}\label{6.27.5}
\mu_n^\d  = \left\{\begin{array}{lll}
\tilde{\mu}_n^\d \left(\|r_n^\d\|^p + \sigma\ep_n\right)^{1-\frac{s}{p}} & \mbox{ if } 
\|r_n^\d\|^p + \sigma \ep_n>(\tau\d)^p,\\[0.8ex]
0 & \mbox{ otherwise}
\end{array}\right.
\end{equation}
with
\begin{equation*}
\tilde{\mu}_n^\d = \min\left\{\frac{\beta_0 \|r_n^\d\|^{p(s-1)}}
{\|L_{i_n}(x_n^\d)^* J_s^{\Y_{i_n}}(r_n^\d)\|^p}, \beta_1\right\}.
\end{equation*}

\item[(iii)] Let $n_\d$ be the first integer such that $ q_{n_\d} =N$ and use $x_{n_\d}^\d$ as an approximate solution.
\end{enumerate}
}
\end{algorithm}

In Algorithm \ref{alg1}, each $\xi_n^\d$ is determined by $F_{i_n}$ completely
without involving $\Theta$, and each $x_n^\d$ is defined by the inexact procedure specified in Assumption
\ref{A2} for solving (\ref{min}) which is independent of $F_i$, $i=0, \cdots, N-1$. This splitting character can make the
implementation of Algorithm \ref{alg1} efficiently. Furthermore, the definition of $x_n^\d$ implies that
\begin{equation*}
\Theta(x_n^\d) -\l \xi_n^\d, x_n^\d\r \le \Theta(x)-\l \xi_n^\d, x\r +\ep_n, \quad \forall x\in \X
\end{equation*}
which shows that
\begin{equation}\label{2014.9.25.1}
\xi_n^\d \in \p_{\ep_n}\Theta(x_n^\d).
\end{equation}
We will use this fact in the forthcoming convergence analysis.

We first prove the following basic result which shows that Algorithm \ref{alg1} is well-defined.

\begin{lemma}\label{lem7.26}
Let $\X$ and $\Y_i$ be reflexive Banach spaces with $\Y_i$ being uniformly smooth, let $\Theta$ and $F_i$, $i=0, \cdots, N-1$
satisfy Assumption \ref{A1}, and let $\{\ep_n\}_{n\ge 0}$ be a sequence of positive numbers satisfying
\begin{equation}\label{2014.9.24.2}
16 \sum_{n=0}^\infty \ep_n \le c_0 \rho^p.
\end{equation}
Let $\beta>1$ be a constant such that $\beta \gamma <1$. If $\beta_0>0$ and $\tau>1$ are
chosen such that
\begin{equation}\label{12.2.1}
c_1:=\frac{1}{\beta} -\gamma -\frac{1+\gamma}{\tau} -\frac{2}{p^*} \left(\frac{\beta_0}{2 c_0}\right)^{p^*-1} >0
\end{equation}
and if $\sigma>0$ is chosen such that $\kappa \beta_1 \sigma\le 1$, where 
\begin{equation}\label{eq:16.1.14}
\kappa = \left\{\begin{array}{lll}
1  & \mbox{ if } p\ge s\\
(\beta^{\frac{p}{s-p}}-1)^{\frac{p-s}{p}} & \mbox{ if } p<s,
\end{array}\right.
\end{equation}
then for Algorithm \ref{alg1} there hold

\begin{enumerate}

\item[(i)] $x_n^\d \in B_{2\rho}(x_0)\cap \D(\Theta)$ for all $n =0, 1, \cdots$;

\item[(ii)] the method terminates after $n_\d<\infty$ iteration steps;

\item[(iii)] for any solution $\hat x$ of (\ref{sys}) in $B_{2\rho}(x_0) \cap \D(\Theta)$ there holds
\begin{equation}\label{eq:11.10}
D_{\xi_{n+1}^\d}^{\ep_{n+1}} \Theta(\hat x, x_{n+1}^\d)-D_{\xi_n^\d}^{\ep_{n}} \Theta (\hat x, x_n^\d)
\le 2 \ep_{n} + \ep_{n+1}
\end{equation}
for all $n\ge 0$. Here we may take $\ep_0=0$ because $\xi_0\in \p \Theta(x_0)$. 
\end{enumerate}
\end{lemma}
\begin{proof}
Let $\hat x$ be any solution of (\ref{sys}) in $B_{2\rho}(x_0)\cap \D(\Theta)$.  We first show that,
if $x_n^\d\in B_{2\rho}(x_0)\cap \D(\Theta)$ for some $n\ge 0$, then
\begin{equation}\label{5.30.1}
D_{\xi_{n+1}^\d}^{\ep_{n+1}}\Theta (\hat x, x_{n+1}^\d)
-D_{\xi_n^\d}^{\ep_n} \Theta(\hat x, x_n^\d) 
\le -c_1 \mu_n^\d \left(\|r_n^\d\|^p+\sigma \ep_n\right)^{\frac{s}{p}} + 2\ep_n +\ep_{n+1}.
\end{equation}
To see this, we consider
\begin{equation*}
\Delta_n:= D_{\xi_{n+1}^\d}^{\ep_{n+1}}\Theta (\hat x, x_{n+1}^\d)
-D_{\xi_n^\d}^{\ep_n} \Theta(\hat x, x_n^\d)
\end{equation*}
which can be written as
\begin{align*}
\Delta_n &= \left[\Theta(x_n^\d)-\l \xi_n^\d, x_n^\d\r -\ep_n\right]
+\left[\l \xi_{n+1}^\d, x_{n+1}^\d\r -\Theta(x_{n+1}^\d)\right]\\
& \quad \, - \l \xi_{n+1}^\d-\xi_n^\d, \hat x\r +\ep_{n+1}.
\end{align*}
Since $\xi_n^\d \in \p_{\ep_n}\Theta(x_n^\d)$, we have from (\ref{4.10.2014}) that
\begin{equation*}
\Theta(x_n^\d)-\l \xi_n^\d, x_n^\d\r -\ep_n \le - \Theta^*(\xi_n^\d).
\end{equation*}
By the definition of $\Theta^*$ we also have
\begin{equation*}
\l \xi_{n+1}^\d, x_{n+1}^\d\r -\Theta(x_{n+1}^\d) \le \Theta^*(\xi_{n+1}^\d).
\end{equation*}
Therefore
\begin{align*}
\Delta_n &\le  \Theta^*(\xi_{n+1}^\d) -\Theta^*(\xi_n^\d)-\l \xi_{n+1}^\d-\xi_n^\d, \hat x\r +\ep_{n+1}\\
&=  \Theta^*(\xi_{n+1}^\d) -\Theta^*(\xi_n^\d)-\l \xi_{n+1}^\d-\xi_n^\d, \nabla \Theta^*(\xi_n^\d)\r \\
& \quad \, + \l \xi_{n+1}^\d-\xi_n^\d, \nabla \Theta^*(\xi_n^\d)-x_n^\d\r \\
& \quad \, +\l \xi_{n+1}^\d-\xi_n^\d, x_n^\d-\hat{x}\r+\ep_{n+1}.
\end{align*}
Because $\Theta$ is $p$-convex, we may use (\ref{Lip}) in Lemma \ref{Fenchel} and the definition of $\xi_{n+1}^\d$ to obtain
\begin{align}\label{2014.9.24}
\Delta_n  & \le \frac{1}{p^*(2 c_0)^{p^*-1}}\|\xi_{n+1}^\d -\xi_n^\d\|^{p^*}
+ \l \xi_{n+1}^\d-\xi_n^\d, \nabla \Theta^*(\xi_n^\d)-x_n^\d\r \nonumber\\
& \quad \, + \mu_n^\d \l J_s^{\Y_{i_n}}(r_n^\d), L_{i_n}(x_n^\d)(\hat x -x_n^\d)\r + \ep_{n+1}.
\end{align}
Since $\xi_n^\d \in \p_{\ep_n} \Theta(x_n^\d)$, we may use Lemma \ref{lem3} to derive that
\begin{equation*}
\l \xi_{n+1}^\d-\xi_n^\d, \nabla \Theta^*(\xi_n^\d)-x_n^\d\r
\le \ep_n + \frac{1}{p^*(2c_0)^{p^*-1}} \|\xi_{n+1}^\d-\xi_n^\d\|^{p^*}.
\end{equation*}
Plugging this estimate into (\ref{2014.9.24}) and using the definition of $\xi_{n+1}^\d$ it follows that
\begin{align*}
\Delta_n & \le \frac{2}{p^*(2c_0)^{p^*-1}} (\mu_n^\d)^{p^*} \|L_{i_n}(x_n^\d)^*
J_s^{\Y_{i_n}}(r_n^\d)\|^{p^*} + \ep_n+ \ep_{n+1} \nonumber\\
& \quad \, + \mu_n^\d \l J_s^{\Y_{i_n}}(r_n^\d), L_{i_n}(x_n^\d)(\hat x-x_n^\d)\r.
\end{align*}
By writing
\begin{equation*}
L_{i_n}(x_n^\d) (\hat x - x_n^\d) = - r_n^\d
-\left[y_{i_n}^\d -F_{i_n}(x_n^\d)-L_{i_n}(x_n^\d) (\hat x-x_n^\d)\right],
\end{equation*}
we may use the condition $\|y_{i_n}^\d-y_{i_n}\|\le \d$, Assumption \ref{A1} (c), and the properties
of $J_s^{\Y_{i_n}}$ to obtain
\begin{align*}
& \l J_s^{\Y_{i_n}}(r_n^\d), L_{i_n}(x_n^\d)(\hat x - x_n^\d)\r \\
&  \le -\|r_n^\d\|^s + \|r_n^\d\|^{s-1} \|y_{i_n}^\d -F_{i_n}(x_n^\d)-L_{i_n}(x_n^\d) (\hat x-x_n^\d)\|\\
& \le -\|r_n^\d\|^s + \|r_n^\d\|^{s-1} \left((1+\gamma)\d + \gamma\|r_n^\d\|\right).
\end{align*}
Therefore 
\begin{align}\label{2014.9.24.1}
\Delta_n & \le \frac{2}{p^*(2c_0)^{p^*-1}} (\mu_n^\d)^{p^*} \|L_{i_n}(x_n^\d)^*
J_s^{\Y_{i_n}}(r_n^\d)\|^{p^*} + \ep_n+ \ep_{n+1} \nonumber\\
& \quad \, - \mu_n^\d \|r_n^\d\|^s + \mu_n^\d \|r_n^\d\|^{s-1} \left((1+\gamma) \d + \gamma \|r_n^\d\|\right).
\end{align}
By the definition of $\mu_n^\d$ we can see that
\begin{align*}
(\mu_n^\d)^{p^*} \|L_{i_n}(x_n^\d)^* J_s^{\Y_{i_n}}(r_n^\d)\|^{p^*}
& = \mu_n^\d (\mu_n^\d)^{p^*-1} \|L_{i_n}(x_n^\d)^* J_s^{\Y_{i_n}}(r_n^\d)\|^{p^*}\\
& \le \beta_0^{p^*-1} \mu_n^\d \|r_n^\d\|^{(s-1)p^*} \left(\|r_n^\d\|^p+\sigma \ep_n\right)^{\frac{(p-s)(p^*-1)}{p}}\\
& \le \beta_0^{p^*-1} \mu_n^\d \left(\|r_n^\d\|^p+\sigma \ep_n\right)^{\frac{s}{p}}
\end{align*}
and 
\begin{align*}
\mu_n^\d \|r_n^\d\|^{s-1}\left((1+\gamma) \d + \gamma \|r_n^\d\|\right) 
& \le \frac{1+\gamma}{\tau} \mu_n^\d \|r_n^\d\|^{s-1}\left(\|r_n^\d\|^p + \sigma\ep_n\right)^{\frac{1}{p}} +\gamma \mu_n^\d \|r_n^\d\|^s\\
& \le \left(\frac{1+\gamma}{\tau} + \gamma\right) \mu_n^\d \left(\|r_n^\d\|^p +\sigma \ep_n\right)^{\frac{s}{p}}.
\end{align*}
Combining the above two estimates with (\ref{2014.9.24.1}) we can obtain
\begin{align}\label{eq:16.1.16}
\Delta_n  & \le \left[\frac{2}{p^*}\left(\frac{\beta_0}{2 c_0}\right)^{p^*-1} +\frac{1+\gamma}{\tau} + \gamma\right]
\mu_n^\d \left(\|r_n^\d\|^p+\sigma\ep_n\right)^{\frac{s}{p}} \nonumber\\
& \quad \, - \mu_n^\d \| r_n^\d\|^s + \ep_n+ \ep_{n+1}.
\end{align}

We next consider the term $\mu_n^\d \|r_n^\d\|^s$. We claim that 
\begin{align}\label{eq:16.1.15}
\mu_n^\d \|r_n^\d\|^s \ge \frac{1}{\beta} \mu_n^\d \left(\|r_n^\d\|^p + \sigma\ep_n\right)^{\frac{s}{p}}
-\kappa \tilde \mu_n^\d \sigma \ep_n,
\end{align}
where $\kappa>0$ is the constant defined by (\ref{eq:16.1.14}). Indeed, this is trivial when $\mu_n^\d =0$. 
We only need to consider the case that $\mu_n^\d \ne 0$.
If $p\ge s$, then
\begin{align*}
\mu_n^\d \|r_n^\d\|^s &= \tilde \mu_n^\d \left(\|r_n^\d\|^p + \sigma \ep_n\right)^{\frac{p-s}{p}} \|r_n^\d\|^s
\ge \tilde \mu_n^\d \|r_n^\d\|^p \\
& = \tilde \mu_n^\d \left(\|r_n^\d\|^p + \sigma\ep_n\right) - \tilde \mu_n^\d \sigma \ep_n\\
& = \mu_n^\d \left(\|r_n^\d\|^p+ \sigma \ep_n\right)^{\frac{s}{p}} -\tilde \mu_n^\d \sigma \ep_n.
\end{align*}
If $p<s$, we may use the inequality
$
(a+b)^t \le \beta a^t + \beta(\beta^{\frac{1}{t-1}}-1)^{1-t} b^t
$
for $a, b\ge 0$ and $t>1$ to derive that
$$
\|r_n^\d\|^s \ge \frac{1}{\beta} \left(\|r_n^\d\|^p+ \sigma \ep_n\right)^{\frac{s}{p}}
-(\beta^{\frac{p}{s-p}}-1)^{\frac{p-s}{p}} (\sigma \ep_n)^{\frac{s}{p}}.
$$
Thus, by using $\mu_n^\d =\tilde \mu_n^\d (\|r_n^\d\|^p+\sigma \ep_n)^{\frac{p-s}{p}}
\le \tilde \mu_n^\d (\sigma \ep_n)^{\frac{p-s}{p}}$, we have
\begin{align*}
\mu_n^\d \|r_n^\d\|^s \ge \frac{1}{\beta} \mu_n^\d \left(\|r_n^\d\|^p+ \sigma \ep_n\right)^{\frac{s}{p}}
-(\beta^{\frac{p}{s-p}}-1)^{\frac{p-s}{p}} \tilde \mu_n^\d \sigma \ep_n.
\end{align*}
We therefore obtain (\ref{eq:16.1.15}).

Combining (\ref{eq:16.1.16}) and (\ref{eq:16.1.15}) we thus have 
\begin{equation*}
\Delta_n \le - c_1 \mu_n^\d \left(\|r_n^\d\|^p+\sigma \ep_n\right)^{\frac{s}{p}} 
+\kappa \tilde \mu_n^\d \sigma \ep_n + \ep_n+ \ep_{n+1},
\end{equation*}
where $c_1>0$ is the constant defined by (\ref{12.2.1}). Since $\tilde \mu_n^\d \le \beta_1$ and 
$\kappa \beta_1 \sigma\le 1$, we therefore obtain (\ref{5.30.1}).

Now we use an induction argument to show that $x_n \in B_{2\rho}(x_0)\cap \D(\Theta)$ for all $n\ge 0$.
This is trivial when $n=0$. Assume that there is some $m\ge 0$ such that
$x_n \in B_{2\rho}(x_0)\cap \D(\Theta)$ for $0\le n<m$. Thus (\ref{5.30.1})
holds for all $0\le n<m$. By taking $\hat x = x^\dag$ in (\ref{5.30.1}) and summing it over these $n$ gives
\begin{equation*}
D_{\xi_m^\d}^{\ep_m} \Theta(x^\dag, x_m) -D_{\xi_0} \Theta(x^\dag, x_0)
\le \sum_{n=0}^{m-1} (2\ep_n+\ep_{n+1}).
\end{equation*}
In view of Lemma \ref{lem2}, (\ref{12.5}) in Assumption \ref{A1} and (\ref{2014.9.24.2}) we can obtain
\begin{align*}
c_0\|x_m-x^\dag\|^p & \le 2 D_{\xi_m^\d}^{\ep_m} \Theta(x^\dag, x_m^\d) + 2\ep_m
\le 2 D_{\xi_0}\Theta(x^\dag, x_0) + 8 \sum_{n=0}^\infty \ep_n\\
& \le \frac{1}{2} c_0 \rho^p + \frac{1}{2} c_0 \rho^p = c_0 \rho^p.
\end{align*}
This implies that $\|x_m^\d -x^\dag\| \le \rho$. By virtue of (\ref{12.5}) and Lemma \ref{lem2} we
also have $\|x_0-x^\dag\| \le \rho$. Therefore $\|x_m^\d -x_0\| \le 2\rho$, i.e. $x_m^\d \in B_{2\rho}(x_0)$.
Consequently, (\ref{5.30.1}) holds for all $n\ge 0$ which gives (\ref{eq:11.10}) immediately. By summing (\ref{5.30.1})
over $n$ from $0$ to $\infty$ we can obtain
\begin{equation}\label{eq:11.13}
c_1 \sum_{n=0}^\infty \mu_n^\d \left(\|r_n^\d\|^p+\sigma \ep_n\right)^{\frac{s}{p}}
\le D_{\xi_0}\Theta(\hat x, x_0) + 3 \sum_{n=0}^\infty \ep_n<\infty.
\end{equation}

Finally we show that $n_\d<\infty$. If it is not true, then for each integer $d\ge 0$ there is at least one
integer $k$ with $Nd \le k\le Nd +N-1$ such that $\|r_k^\d\|^p +\sigma \ep_k>(\tau \d)^p$. Consequently
\begin{equation*}
\mu_k^\d = \tilde\mu_k^\d \left(\|r_k^\d\|^p+\sigma\ep_k\right)^{1-\frac{s}{p}} 
\ge c_2 \left(\|r_k^\d\|^p+\sigma \ep_k\right)^{1-\frac{s}{p}}
\end{equation*}
with $c_2 = \min\{\beta_0 B^{-p}, \beta_1\}$, where we used (\ref{eq:L}). Therefore
\begin{equation*}
\sum_{n=Nd}^{Nd+N-1} \mu_n^\d \left(\|r_n^\d\|^p+\sigma \ep_n\right)^{\frac{s}{p}} 
\ge \mu_k^\d \left(\|r_k^\d\|^p+\sigma \ep_k\right)^{\frac{s}{p}} \ge c_2 \left(\|r_k^\d\|^p+\sigma \ep_k\right)
\ge c_2 (\tau \d)^p
\end{equation*}
and hence
\begin{align*}
\sum_{n=0}^\infty \mu_n^\d \left(\|r_n^\d\|^p +\sigma \ep_n\right)^{\frac{s}{p}} 
= \sum_{d=0}^\infty \sum_{n=Nd}^{Nd+N-1} \mu_n^\d \left(\|r_n^\d\|^p+\sigma\ep_n\right)^{\frac{s}{p}}
\ge c_2 \sum_{d=0}^\infty (\tau \d)^p =\infty
\end{align*}
which is a contradiction to (\ref{eq:11.13}).
\end{proof}

\subsection{The method with exact data}

In order to prove the regularization property of Algorithm \ref{alg1}, we first consider its
counterpart where the noisy data $y_i^\d$ are replaced by the exact data $y_i$. That is, we will
consider the following algorithm.

\begin{algorithm}[Landweber-Kaczmarz method with exact data] \label{alg2} \quad

\noindent
{\it
Let $\beta_0>0$, $\beta_1>0$, $\sigma>0$, $\{\ep_n\}$, $x_0\in \X$ and $\xi_0\in \X^*$ be the same as in
Algorithm \ref{alg1}. For $n=0, 1, \cdots$ we define $r_n = F_{i_n}(x_n) -y_{i_n}$ and update 
\begin{equation}\label{6.27}
\begin{array}{lll}
\xi_{n+1}  =\xi_n-\mu_n L_{i_n}(x_n)^* J_s^{\Y_{i_n}}(r_n),\\[1ex]
x_{n+1}  = S_{\ep_{n+1}}(\xi_{n+1}),
\end{array}
\end{equation}
where $i_n = n \, (\mbox{mod } N)$ and
\begin{equation}\label{6.26}
\mu_n  =\tilde{\mu}_n \left(\|r_n\|^p+\sigma \ep_n\right)^{1-\frac{s}{p}} 
\end{equation}
with
$$
\tilde{\mu}_n = \min\left\{\frac{\beta_0 \|r_n\|^{p(s-1)}}
{\|L_{i_n}(x_n)^* J_s^{\Y_{i_n}}(r_n)\|^p}, \beta_1\right\}.
$$
}
\end{algorithm}

By using the same argument in the proof of Lemma \ref{lem7.26}, we can obtain the following result on
Algorithm \ref{alg2}.

\begin{lemma}\label{lem11.11}
Let $\X$ and $\Y_i$ be reflexive Banach spaces with $\Y_i$ being uniformly smooth, let $\Theta$ and $F_i$, $i=0, \cdots, N-1$
satisfy Assumption \ref{A1}, and let $\{\ep_n\}$ be a sequence of positive numbers satisfying (\ref{2014.9.24.2}).
Let $\beta>1$ be a constant such that $\beta \gamma <1$. If $\beta_0>0$ is chosen such that
\begin{equation}\label{11.2}
\frac{1}{\beta} -\gamma -\frac{2}{p^*} \left(\frac{\beta_0}{2 c_0}\right)^{p^*-1} >0
\end{equation}
and if $\sigma>0$ is chosen such that $\kappa \beta_1 \sigma \le 1$, where $\kappa$ is defined by (\ref{eq:16.1.14}),
then for Algorithm \ref{alg2} there holds $x_n \in B_{2\rho}(x_0)\cap \D(\Theta)$ for all $n=0, 1, \cdots$,
and for any solution $\hat x$ of (\ref{sys}) in $B_{2\rho}(x_0) \cap \D(\Theta)$ there hold
\begin{align}
& D_{\xi_{n+1}}^{\ep_{n+1}} \Theta(\hat x, x_{n+1})-D_{\xi_n}^{\ep_{n}} \Theta (\hat x, x_n)
\le 2\ep_{n} + \ep_{n+1}, \quad n =0, 1, \cdots \label{eq:11.3}\\
& \sum_{n=0}^\infty \mu_n \left(\|r_n\|^p + \sigma \ep_n\right)^{\frac{s}{p}} <\infty. \label{eq:11.3.1}
\end{align}
\end{lemma}

By the definition of $\mu_n$ and (\ref{eq:L}) one can see that
\begin{equation*}
\mu_n \left(\|r_n\|^p+\sigma \ep_n\right)^{\frac{s}{p}} 
\ge \min\{\beta_0 B^{-p}, \beta_1\} \left(\|r_n\|^p +\sigma \ep_n\right).
\end{equation*}
Thus, it follows from (\ref{eq:11.3.1}) that
\begin{equation}\label{eq:11.3.2}
\sum_{n=0}^\infty \left(\|r_n\|^p + \ep_n\right)<\infty.
\end{equation}

In the following we will show that the sequence $\{x_n\}$ defined by Algorithm \ref{alg2} converges
to a solution of (\ref{sys}). We first prove that the sequence $\{D_{\xi_n}^{\ep_{n}} \Theta(\hat x, x_n)\}$
is convergent.

\begin{lemma}\label{cor16}
Let all the conditions in Lemma \ref{lem11.11} hold. Then, for any solution $\hat x$ of (\ref{sys})
in $B_{2\rho}(x_0)\cap \D(\Theta)$, the sequence $\{D_{\xi_n}^{\ep_{n}} \Theta(\hat x, x_n)\}$ is convergent.
\end{lemma}

\begin{proof}
Let
\begin{equation*}
a_n := D_{\xi_n}^{\ep_{n}} \Theta(\hat x, x_n) \quad \mbox{and} \quad
\beta_n := 2\ep_{n} + \ep_{n+1}.
\end{equation*}
It follows from (\ref{eq:11.3}) in Lemma \ref{lem11.11} and $\sum_n \ep_n<\infty$ that
\begin{equation*}
0\le a_{n+1} \le a_n + \beta_n,  \quad n=0, 1, \cdots
\end{equation*}
with $\sum_{n=0}^\infty \beta_n <\infty$. It is easily seen that
\begin{equation*}
0\le a_n \le a_0 +\sum_{k=0}^{n-1} \beta_k \le a_0 +\sum_{k=0}^\infty \beta_k <\infty,
\end{equation*}
i.e. $\{a_n\}$ is a bounded sequence. Thus $\limsup_{n\rightarrow \infty} a_n$ exists and is finite. Since for $k<n$ we have
$a_n \le a_k +\sum_{j=k}^{n-1} \beta_j$, we may derive that
$
\limsup_{n\rightarrow \infty} a_n \le a_k +\sum_{j=k}^\infty \beta_j.
$
This then implies that
\begin{equation*}
\liminf_{k\rightarrow \infty} a_k \ge \limsup_{n\rightarrow \infty} a_n  -\lim_{k \rightarrow \infty} \sum_{j=k}^\infty \beta_j
=\limsup_{n\rightarrow \infty} a_n.
\end{equation*}
Therefore $\liminf_{n\rightarrow \infty} a_n =\limsup_{n\rightarrow \infty} a_n$, i.e. $\lim_{n\rightarrow \infty} a_n$ exists.
\end{proof}

\begin{lemma}\label{lem11.13}
For all $n\ge 0$ there holds
\begin{equation*}
\|x_n-\nabla \Theta^*(\xi_n)\|^p \le \frac{p}{2 c_0} \ep_n.
\end{equation*}
\end{lemma}

\begin{proof}
Recall $\xi_n \in \p_{\ep_n} \Theta(x_n)$. This result follows from Lemma \ref{lem3} immediately.
\end{proof}

\begin{lemma}\label{lem2015.11.12}
Let all the conditions in Lemma \ref{lem11.11} hold. Then there is a universal constant $C$ such that
\begin{equation*}
\|F_i(x_n)-y_i\|^p \le C \sum_{k=\min\{n, n+i-i_n\}}^{\max\{n, n+i-i_n\}} \left(\|r_k\|^p+ \ep_k\right)
\end{equation*}
for all $n\ge 0$ and $i=0, \cdots, N-1$, where $i_n = n \, (\mbox{mod } N)$.
\end{lemma}

\begin{proof}
We consider the case $i\ge i_n$; the case $i<i_n$ can be considered similarly. Let $d = (n-i_n)/N$.
By the triangle inequality and (\ref{12.6.1}) we have
\begin{align*}
\|F_i(x_n)-y_i\|  &= \|F_i(x_{Nd+i_n})-y_i\| \\
& \le \|F_i(x_{Nd+i})-y_i\| + \|F_i(x_{Nd+i})-F_i(x_{Nd+i_n})\|\\
& \le \|r_{Nd+i}\| + \sum_{j=i_n}^{i-1} \|F_i(x_{Nd+j+1})-F_i(x_{Nd+j})\| \nonumber\\
& \le \|r_{Nd+i}\| + \frac{B}{1-\gamma} \sum_{j=i_n}^{i-1} \|x_{Nd+j+1}-x_{Nd+j}\|.
\end{align*}
By virtue of the H\"{o}lder inequality and $i-i_n \le N-1$, we can obtain
\begin{align}\label{eq:11.7}
& \|F_i(x_n)-y_i\|^p  \nonumber\\
& \le N^{p-1} \left(\|r_{Nd+i}\|^p + \frac{B^p}{(1-\gamma)^p} \sum_{j=i_n}^{i-1} \|x_{Nd+j+1}-x_{Nd+j}\|^p\right).
\end{align}
In view of Lemma \ref{lem11.13} and (\ref{10.13.4}), we have
\begin{align*}
& \|x_{Nd+j+1}-x_{Nd+j}\|^p\\
&\le 3^{p-1}\big(\|x_{Nd+j+1}-\nabla \Theta^*(\xi_{Nd+j+1})\|^p + \|x_{Nd+j} -\nabla \Theta^*(\xi_{Nd+j})\|^p\\
&\quad \, + \|\nabla \Theta^*(\xi_{Nd+j+1})-\nabla \Theta^*(\xi_{Nd+j})\|^p\big) \\
&\le \frac{p 3^{p-1}}{2c_0} (\ep_{Nd+j}+\ep_{Nd+j+1}) + 3^{p-1}\left(\frac{\|\xi_{Nd+j+1}-\xi_{Nd+j}\|}{2 c_0}\right)^{p^*}.
\end{align*}
By the definition of $\xi_{Nd+j+1}$ and $\mu_{Nd+j}$ we have
\begin{align*}
\|\xi_{Nd+j+1}-\xi_{Nd+j}\| & = \mu_{Nd+j} \|L_j(x_{Nd+j})^* J_s^{\Y_j}(r_{Nd+j})\|\\
&\le B \mu_{Nd+j} \|r_{Nd+j}\|^{s-1} \\
&\le B \beta_1 \left(\|r_{Nd+j}\|^p+\sigma \ep_{Nd+j}\right)^{1-\frac{1}{p}}.
\end{align*}
Thus
\begin{align*}
& \|x_{Nd+j+1}-x_{Nd+j}\|^p\\
& \le \frac{p3^{p-1}}{2c_0} (\ep_{Nd+j}+\ep_{Nd+j+1}) 
+ 3^{p-1}\left(\frac{B \beta_1}{2 c_0}\right)^{p^*} \left(\|r_{Nd+j}\|^p+\sigma \ep_{Nd+j}\right).
\end{align*}
Combining this with (\ref{eq:11.7}) gives
\begin{align*}
\|F_i(x_n)-y_i\|^p 
& \le C\Big(\|r_{Nd+i}\|^p + \sum_{j=i_n}^{i-1} (\|r_{Nd+j}\|^p + \ep_{Nd+j+1}+\ep_{Nd+j})\Big)\\
& \le C \sum_{k=n}^{n+i-i_n} \left(\|r_k\|^p + \ep_k\right),
\end{align*}
where $C$ is a universal constant. This completes the proof.
\end{proof}

\begin{corollary}\label{cor1}
Let all the conditions in Lemma \ref{lem11.11} hold. Then
\begin{equation*}
\lim_{n\rightarrow \infty} \|F_i(x_n)-y_i\| =0
\end{equation*}
for all $i=0, \cdots, N-1$.
\end{corollary}

\begin{proof}
This result follows from Lemma \ref{lem2015.11.12} and (\ref{eq:11.3.2}).
\end{proof}

Now we are ready to prove the convergence on the sequence $\{x_n\}$ defined by Algorithm \ref{alg2}.

\begin{theorem}\label{thm:exact}
Let all the conditions in Lemma \ref{lem11.11} hold. Then for Algorithm \ref{alg2} there exists a
solution $x_*\in B_{2\rho}(x_0)\cap \D(\Theta)$ of (\ref{sys}) such that
$$
\lim_{n\rightarrow \infty} \|x_n-x_*\|=0 \quad \mbox{ and } \quad
\lim_{n\rightarrow \infty} D_{\xi_n}^{\ep_n} \Theta(x_*, x_n) =0.
$$
\end{theorem}

\begin{proof}
We first show that $\{x_n\}$ has a convergent subsequence. To this end, we consider.
\begin{equation*}
R_d := \sum_{n= Nd}^{Nd+N-1} \left(\|r_n\|^p+\ep_n\right), \quad d =0, 1, \cdots.
\end{equation*}
Then $R_d>0$ and, in view of (\ref{eq:11.3.2}), we have
$R_d\rightarrow 0$ as $d\rightarrow \infty$. We may choose a strictly increasing subsequence $\{d_k\}$
of integers such that $d_0=0$ and $d_k$, for each $k \ge 1$, is the first integer satisfying
\begin{equation*}
d_k \ge d_{k-1}+1  \quad \mbox{ and } \quad R_{d_k} \le R_{d_{k-1}}.
\end{equation*}
For this $\{d_k\}$ it can be seen that
\begin{equation}\label{eq81}
R_d \ge R_{d_k} \qquad \forall d\le d_k.
\end{equation}
Indeed, for any $d$ satisfying $0\le d<d_k$, we can find $0\le l< k$ such that $d_l \le d < d_{l+1}$ and thus,
by the definition of $d_{l+1}$, we have $R_d\ge R_{d_l} \ge R_{d_k}$.

With the above chosen $\{d_k\}$, we set $\{n_k\} := \{N d_k\}$ and show that $\{x_{n_k}\}$ is convergent.
To this end, we consider the $\ep$-Bregman distance $D_{\xi_{n_k}}^{\ep_{n_k}} \Theta(x_{n_l}, x_{n_k})$
for $k<l$. Let $\hat x$ be any solution of (\ref{sys}) in $B_{2\rho}(x_0) \cap \D(\Theta)$.
By the definition of $\ep$-Bregman distance, we have
\begin{align}\label{eq12.1}
D_{\xi_{n_k}}^{\ep_{n_k}} \Theta(x_{n_l}, x_{n_k})
& = D_{\xi_{n_k}}^{\ep_{n_k}} \Theta(\hat x, x_{n_k}) -D_{\xi_{n_l}}^{\ep_{n_l}} \Theta(\hat x, x_{n_l}) \nonumber\\
& \quad \, + \l \xi_{n_l}-\xi_{n_k}, x_{n_l}-\hat x\r +\ep_{n_l}.
\end{align}
We need to consider the term $\l \xi_{n_l}-\xi_{n_k}, x_{n_l}-\hat x\r$ for $k<l$. We write
\begin{align*}
\l \xi_{n_l}-\xi_{n_k}, x_{n_l}-\hat x\r &=\l \xi_{N d_l}-\xi_{N d_k}, x_{Nd_l}-\hat x\r \\
& =\sum_{d=d_k}^{d_l-1} \l \xi_{N(d+1)}-\xi_{Nd}, x_{Nd_l}-\hat x\r.
\end{align*}
According to the definition of $\xi_{n+1}$ we have
\begin{align*}
\l \xi_{N(d+1)}-\xi_{Nd}, x_{Nd_l}-\hat x\r
&= \sum_{n=Nd}^{Nd+N-1} \l \xi_{n+1}-\xi_n, x_{N d_l}-\hat x\r\\
& = \sum_{n=Nd}^{Nd+N-1} \mu_n \l J_s^{\Y_{i_n}}(r_n), L_{i_n}(x_n)(x_{Nd_l}-\hat x)\r.
\end{align*}
By the Cauchy-Schwarz inequality, the property of $J_s^{\Y_{i_n}}$ and the definition of $\mu_n$ we obtain
\begin{align*}
\left|\l \xi_{N(d+1)}-\xi_{Nd}, x_{Nd_l}-\hat x\r\right|
& \le \sum_{n=Nd}^{Nd+N-1} \mu_n \|r_n\|^{s-1} \|L_{i_n}(x_n)(x_{Nd_l}-\hat x)\|\\
& \le \beta_1 \sum_{n=Nd}^{Nd+N-1} \left(\|r_n\|^p + \sigma \ep_n\right)^{1-\frac{1}{p}} \|L_{i_n}(x_n)(x_{Nd_l}-\hat x)\|.
\end{align*}
By Assumption \ref{A1} (c) we have
\begin{align*}
\| L_{i_n}(x_n)(x_{Nd_l}-\hat x)\|
& \le \| L_{i_n}(x_n) (x_n-\hat x)\| + \|L_{i_n}(x_n) (x_{Nd_l}-x_n)\|\\
& \le (1+\gamma) \left(\|F_{i_n}(x_n)-F_{i_n}(\hat x)\| + \|F_{i_n}(x_{Nd_l}) - F_{i_n}(x_n)\|\right) \\
& \le (1+\gamma) \left( 2\|r_n\|+ \|F_{i_n}(x_{N d_l})-y_{i_n}\|\right).
\end{align*}
Combining the above two equations and using the H\"{o}lder inequality, we can find a universal constant $C$ such that 
\begin{align*}
& \left|\l \xi_{N(d+1)}-\xi_{Nd}, x_{N d_l}-\hat x\r\right|\\
& \le C R_d + C \sum_{n=Nd}^{Nd+N-1} \left(\|r_n\|^p+\ep_n\right)^{1-\frac{1}{p}} \|F_{i_n}(x_{N d_l})-y_{i_n}\|\\
& \le C R_d + C R_d^{1/p^*} \left(\sum_{n=Nd}^{Nd+N-1} \|F_{i_n}(x_{N d_l})-y_{i_n}\|^p\right)^{1/p}.
\end{align*}
By using Lemma \ref{lem2015.11.12} we have
\begin{align*}
\sum_{n=Nd}^{Nd+N-1} \|F_{i_n}(x_{Nd_l})-y_{i_n}\|^p
& \le C \sum_{n=Nd}^{Nd+N-1} \sum_{k=N d_l}^{Nd_l+i_n} \left(\|r_k\|^p + \ep_k\right)\\
& \le C N \sum_{k=Nd_l}^{Nd_l+N-1} \left(\|r_k\|^p + \ep_k\right)\\
& = C N R_{d_l}.
\end{align*}
Consequently, it follows from (\ref{eq81}) that
\begin{equation*}
\left|\l \xi_{N(d+1)}-\xi_{Nd}, x_{N d_l}-\hat x\r\right| \le C R_d + C R_d^{1/p^*} R_{d_l}^{1/p} \le C R_d.
\end{equation*}
Therefore
\begin{align*}
|\l \xi_{n_l}-\xi_{n_k}, x_{n_l}-\hat x\r|
\le  C \sum_{d=d_k}^{d_l-1} R_d = C \sum_{n=n_k}^{n_l-1} \left(\|r_n\|^p  + \ep_n\right).
\end{align*}
This together with (\ref{eq:11.3.2}) implies that
\begin{equation}\label{5.1}
\lim_{k \rightarrow \infty} \sup_{l\ge k} \left|\l \xi_{n_l}-\xi_{n_k}, x_{n_l}-\hat x\r\right| = 0.
\end{equation}
Thus, it follows from (\ref{eq12.1}), Lemma \ref{cor16} and $\lim_{n\rightarrow \infty}\ep_n=0$
that
\begin{equation*}
D_{\xi_{n_k}}^{\ep_{n_k}} \Theta (x_{n_l}, x_{n_k}) \rightarrow 0 \quad \mbox{ as } k, l\rightarrow \infty.
\end{equation*}
In view of Lemma \ref{lem2} and $\lim_{n\rightarrow \infty} \ep_n=0$, we can conclude
\begin{equation*}
\|x_{n_l}-x_{n_k}\|\rightarrow 0 \qquad \mbox{as } k, l\rightarrow \infty,
\end{equation*}
i.e. $\{x_{n_k}\}$ is a Cauchy sequence in $\X$. Thus $x_{n_k}\rightarrow x_*$ as $k\rightarrow \infty$
for some $x_*\in \X$. By using Corollary \ref{cor1} and the continuity of $F_i$ we have
$F_i(x_*) = y_i$ for all $i=0, \cdots, N-1$.

We next show that
$
x_*\in B_{2\rho}(x_0) \cap \D(\Theta).
$
Since $\{x_{n_k}\}\subset B_{2\rho}(x_0)$,
 we must have $x_*\in B_{2\rho}(x_0)$. By using $\xi_{n_k}\in \p_{\ep_{n_k}} \Theta(x_{n_k})$ we can obtain
\begin{equation}\label{5.5.3.1}
\Theta(x_{n_k})\le \Theta(\hat{x}) +\l \xi_{n_k}, x_{n_k}-\hat{x}\r +\ep_{n_k}.
\end{equation}
In view of (\ref{5.1}) and $x_{n_k}\rightarrow x_*$, there is a constant $C_0$ such that
\begin{equation*}
|\l \xi_{n_k} -\xi_{n_0}, x_{n_k}-\hat x\r| \le C_0 \quad \mbox{and} \quad |\l \xi_{n_0}, x_{n_k}-\hat x\r| \le C_0, \quad \forall k.
\end{equation*}
Thus $|\l \xi_{n_k}, x_{n_k}-\hat{x}\r|\le 2C_0$ for all $k$. By using the lower semi-continuity of $\Theta$ and
$\lim_{n\rightarrow \infty} \ep_n=0$ we obtain from (\ref{5.5.3.1}) that
\begin{equation*}
\Theta(x_*)\le \liminf_{k\rightarrow\infty} \Theta(x_{n_k})\le \Theta(\hat{x}) + 2C_0 <\infty.
\end{equation*}
This implies that $x_*\in \D(\Theta)$. Therefore $x_*$ is a solution of (\ref{sys}) in $B_{2\rho}(x_0)\cap \D(\Theta)$.

Finally we show that the whole sequence $\{x_n\}$ converges to $x_*$. Let
\begin{equation*}
\eta_0 := \lim_{n\rightarrow \infty} D_{\xi_n}^{\ep_n} \Theta(x_*, x_n)
\end{equation*}
whose existence is guaranteed by Lemma \ref{cor16}. By the non-negativity of $\ep$-Bregman distance,
we have $\eta_0\ge 0$. In (\ref{eq12.1}) we set $\hat x = x_*$ and take $l\rightarrow \infty$ to derive that
\begin{align*}
D_{\xi_{n_k}}^{\ep_{n_k}} \Theta(x_*, x_{n_k})
\le D_{\xi_{n_k}}^{\ep_{n_k}} \Theta(x_*, x_{n_k}) -\eta_0 + \sup_{l\ge k} \left|\l \xi_{n_l}-\xi_{n_k}, x_{n_l}-x_*\r\right|.
\end{align*}
This implies that
\begin{equation*}
\eta_0 \le \sup_{l\ge k} \left|\l \xi_{n_l}-\xi_{n_k}, x_{n_l}-x_*\r\right|
\end{equation*}
for all $k$. In view of (\ref{5.1}), by taking $k\rightarrow \infty$ we obtain $\eta_0\le 0$. Therefore $\eta_0=0$, that is,
$\lim_{n\rightarrow \infty} D_{\xi_n}^{\ep_n} \Theta(x_*, x_n) =0$. By using Lemma \ref{lem2}
and $\lim_{n\rightarrow \infty} \ep_n=0$ we can obtain $\lim_{n\rightarrow \infty} \|x_n-x_*\|=0$.
\end{proof}

\subsection{Regularization property}

We return to Algorithm \ref{alg1} and prove its regularization property. We need the following
stability result.

\begin{lemma}\label{stability0}
Let all the conditions in Lemma \ref{lem7.26} hold. Then for all $n\ge 0$ there hold
\begin{equation*}
\xi_n^{\d} \rightarrow \xi_n \quad \mbox{ and } \quad x_n^{\d}\rightarrow x_n
\quad \mbox{as } \d\rightarrow 0.
\end{equation*}
\end{lemma}

\begin{proof}
The result is trivial for $n=0$. We next assume that the result is true for some $n\ge 0$
and show that $\xi_{n+1}^{\d}\rightarrow \xi_{n+1}$ and $x_{n+1}^{\d}\rightarrow x_{n+1}$
as $\d\rightarrow 0$. We consider two cases.

{\it Case 1: $F_{i_n}(x_n)=y_{i_n}$}. In this case we have $\xi_{n+1} =\xi_n$. Therefore
\begin{equation*}
\xi_{n+1}^{\d}-\xi_{n+1} =\xi_n^{\d}-\xi_n -\mu_n^{\d} L_{i_n}(x_n^\d)^*
J_s^{\Y_{i_n}} (r_n^\d).
\end{equation*}
Since $\mu_n^{\d} \le \beta_1 (\|r_n^\d\|^p+\sigma\ep_n)^{1-\frac{s}{p}}$, we may use the property
of $J_s^{\Y_{i_n}}$ and (\ref{eq:L}) to obtain
\begin{equation*}
\|\xi_{n+1}^{\d}-\xi_{n+1}\| \le\|\xi_n^{\d}-\xi_n\| +\beta_1 B \left(\|r_n^\d\|^p+\sigma \ep_n\right)^{1-\frac{s}{p}}
\|r_n^\d\|^{s-1}.
\end{equation*}
By the induction hypothesis and the continuity of $F_{i_n}$, we then have $\xi_{n+1}^{\d}\rightarrow \xi_{n+1}$
as $\d\rightarrow 0$. By the definition of $x_{n+1}^\d$ and the continuity of $S_{\ep}$ we have
\begin{equation*}
x_{n+1}^{\d} = S_{\ep_{n+1}}(\xi_{n+1}^{\d}) \rightarrow S_{\ep_{n+1}}(\xi_{n+1})=x_{n+1}
\end{equation*}
as $\d\rightarrow 0$.

{\it Case 2: $F_{i_n}(x_n)\ne y_{i_n}$}. Since $\ep_n>0$, we have $\|r_n^\d\|^p+\sigma\ep_n >(\tau \d)^p$ for small $\d>0$. Thus
\begin{equation*}
\mu_n = \tilde \mu_n \left(\|r_n\|^p+\sigma \ep_n\right)^{1-\frac{s}{p}} \quad \mbox{ and } \quad
\mu_n^\d = \tilde \mu_n^\d \left(\|r_n^\d\|^p+\sigma \ep_n\right)^{1-\frac{s}{p}}.
\end{equation*}
We claim that $\mu_n^{\d} \rightarrow \mu_n$ as $\d\rightarrow 0$. In fact, if
$
L_{i_n}(x_n)^* J_s^{\Y_{i_n}} (r_n) = 0,
$
then, by definition, we must have $\tilde \mu_n=\beta_1$ and $\tilde \mu_n^{\d} =\beta_1$ for small $\d>0$. If
$
L_{i_n}(x_n)^* J_s^{\Y_{i_n}}(r_n) \ne 0,
$
then, by the definition of $\tilde\mu_n^\d$ and $\tilde\mu_n$ and the induction hypothesis, we can
conclude that $\tilde \mu_n^{\d}\rightarrow \tilde \mu_n$ as $\d\rightarrow 0$. In any case,
we always have $\tilde \mu_n^{\d} \rightarrow \tilde \mu_n$ as $\d\rightarrow 0$. Consequently,
by the induction hypotheses and the continuity of $F_{i_n}$, $L_{i_n}$ and $J_s^{\Y_{i_n}}$, we have
$\mu_n^{\d}\rightarrow \mu_n$ and therefore $\xi_{n+1}^{\d}\rightarrow \xi_{n+1}$ as $\d\rightarrow 0$.
By invoking again the continuity of $S_{\ep}$, we obtain $x_{n+1}^{\d}
\rightarrow x_{n+1}$ as $\d\rightarrow 0$.
\end{proof}

\begin{theorem}\label{T6.4}
Assume that all the conditions in Lemma \ref{lem7.26} hold. Then for Algorithm \ref{alg1} there hold
\begin{equation*}
\lim_{\d\rightarrow 0} \|x_{n_\d}^\d-x_*\|=0 \qquad \mbox{and} \qquad
\lim_{\d\rightarrow 0} D_{\xi_{n_\d}^\d}^{\ep_{n_\d}} \Theta(x_*, x_{n_\d}^\d) =0,
\end{equation*}
where $x_*$ is a solution of (\ref{sys}) in $B_{2\rho}(x_0) \cap \D(\Theta)$.
\end{theorem}

\begin{proof}
By the definition of $n_\d$ we have $\sigma \ep_{n_\d} \le (\tau \d)^p$. Since $\ep_n>0$ for all $n$, 
we must have $n_\d\rightarrow \infty$ as $\d\rightarrow 0$.
Let $x_*$ be the solution of (\ref{sys}) in $B_{2\rho}(x_0)\cap \D(\Theta)$ determined in Theorem \ref{thm:exact}.
In view of Lemma \ref{p-conv} and $\lim_{n\rightarrow \infty} \ep_n=0$,  it suffices to show that
\begin{equation}\label{9.26.31}
\lim_{\d\rightarrow 0} D_{\xi_{n_\d}^\d}^{\ep_{n_\d}} \Theta(x_*, x_{n_\d}^\d)=0.
\end{equation}
Let $n$ be an arbitrary but fixed integer. By virtue of Lemma \ref{lem7.26} we have
\begin{eqnarray*}
D_{\xi_{n_\d}^\d}^{\ep_{n_\d}}\Theta(x_*, x_{n_\d}^\d)
& \le D_{\xi_n^\d}^{\ep_n}\Theta(x_*, x_n^\d)
+ 3 \sum_{k=n}^{n_\d} \ep_k.
\end{eqnarray*}
Consequently
\begin{eqnarray*}
& \limsup_{\d\rightarrow 0} D_{\xi_{n_\d}^\d}^{\ep_{n_\d}} \Theta (x_*, x_{n_\d})
\le \limsup_{\d\rightarrow 0} D_{\xi_n^\d}^{\ep_n}\Theta(x_*, x_n^\d)
+ 3 \sum_{k=n}^\infty \ep_k.
\end{eqnarray*}
By making use of Lemma \ref{stability0} and the lower semi-continuity of $\Theta$ we have
\begin{equation*}
\limsup_{\d\rightarrow 0} D_{\xi_n^\d}^{\ep_n}\Theta(x_*, x_n^\d) \le D_{\xi_n}^{\ep_n} \Theta(x_*, x_n).
\end{equation*}
Therefore
\begin{equation*}
\limsup_{\d\rightarrow 0} D_{\xi_{n_\d}^\d}^{\ep_{n_\d}} \Theta (x_*, x_{n_\d})
\le D_{\xi_n}^{\ep_n}\Theta(x_*, x_n) + 3 \sum_{k=n}^\infty \ep_k.
\end{equation*}
Since $n$ can be arbitrarily large, by taking $n\rightarrow \infty$ we can derive from Theorem \ref{thm:exact}
and the condition $\sum_{k=0}^\infty \ep_k <\infty$ that
\begin{equation*}
\limsup_{\d\rightarrow 0} D_{\xi_{n_\d}^\d}^{\ep_{n_\d}} \Theta (x_*, x_{n_\d}) \le 0.
\end{equation*}
This completes the proof.
\end{proof}

\subsection{Acceleration}

It is well-known that Landweber iteration is a slowly convergent method (\cite{EHN96}). In order to make it applicable
in practical applications, acceleration strategies should be incorporated into the method. When an inverse problem
is formulated in Hilbert spaces, a family of accelerated Landweber iterations, including the famous $\nu$-method
of Brakhage [2], have been proposed in [3] using the orthogonal polynomials and the spectral theory of
self-adjoint operators. The acceleration strategy using orthogonal polynomials is no longer available
when an inverse problems is considered in Banach spaces using general convex penalty functions. Instead the
sequential subspace optimization strategy has been employed in \cite{HJW2015,SS2009} to accelerate the method.

In recent years Nesterov's acceleration strategy \cite{Nest1983} has received tremendous consideration in
optimization community. Consider the unconstrained optimization problem
\begin{equation}\label{11.20.1}
\min_{x\in \X} \, \varphi(x)
\end{equation}
in a Hilbert space $\X$ with a continuous differentiable function $\varphi: \X \to {\mathbb R}$. Nesterov's
strategy speeds up the gradient descent method by using a proper extrapolation point at each iteration
step. It takes the form
\begin{equation}\label{11.20.2}
\left\{\begin{array}{lll}
\hat x_n = x_n + {\frac{n}{n+\a} (x_n -x_{n-1})},\\[1ex]
x_{n+1} = \hat x_n - t_n \nabla \varphi(\hat x_n)
\end{array}\right.
\end{equation}
with suitable step sizes $t_n>0$, where $\a\ge 3$ is a fixed number. This strategy has been extended in
various context, even for nonsmooth optimization problems, see \cite{AP2015,BT2009,CD2014}.

Due to the simplicity of Nesterov's strategy, it is natural to consider its use in accelerating our
Landweber-Kaczmarz method. We propose the following accelerated version in which we drop the superscript
$\d$ for all the iterates for simplicity.

\begin{algorithm} \label{alg3}
{\it 
Let $\beta_0>0$, $\beta_1>0$, $\sigma>0$ and $\tau>1$ be suitably chosen numbers, let $\{\ep_n\}$ be a 
summable sequence of positive numbers, and let $\a\ge 3$ be a fixed number.

\begin{enumerate}

\item[(i)] Pick $x_0\in \X$ and $\xi_0\in \X^*$ such that $\xi_0\in \p \Theta(x_0)$. 

\item[(ii)] Let $\xi_{-1} = \xi_0$, $x_{-1} = x_0$, $q_{-1} =0$ and $\hat x_0 = x_0$. For $n\ge 0$ we define
$\hat r_n = F_{i_n}(\hat x_n) - y_{i_n}^\d$ and 
\begin{equation*}
q_{n}  = \left\{\begin{array}{lll}
q_{n-1}+1 \ & \mbox{ if } \|\hat r_n\|^p + \sigma \ep_n \le (\tau\d)^p,\\[0.8ex]
0 & \mbox{ otherwise}.
\end{array}\right.
\end{equation*}
We then update 
\begin{equation}\label{11.19.0}
\left\{\begin{array}{lll}
\hat \xi_n = \xi_n + \frac{n}{n+\a} (\xi_n-\xi_{n-1}),\\[1.2ex]
\hat x_n = x_n +\frac{n}{n+\a} (x_n-x_{n-1}),\\[1.2ex]
\xi_{n+1}  =\hat \xi_n - \mu_n L_{i_n}(\hat x_n)^* J_s^{\Y_{i_n}}\left(\hat r_n\right),\\[1.2ex]
x_{n+1}  = S_{\ep_{n+1}}(\xi_{n+1}),
\end{array}\right.
\end{equation}
where
\begin{equation*}
\mu_n  = \left\{\begin{array}{lll}
\tilde{\mu}_n \left(\|\hat r_n\|^p+\sigma \ep_n\right)^{1-\frac{s}{p}} \ & \mbox{ if } \|\hat r_n\|^p +\sigma \ep_n >(\tau\d)^p,\\[0.8ex]
0 & \mbox{ otherwise}
\end{array}\right.
\end{equation*}
with
\begin{equation*}
\tilde{\mu}_n = \min\left\{\frac{\beta_0 \|\hat r_n\|^{p(s-1)}}
{\|L_{i_n}(\hat x_n)^* J_r^{\Y_{i_n}}(\hat r_n)\|^p}, \beta_1\right\};
\end{equation*}

\item[(iii)] Let $n_\d$ be the first integer such that $q_{n_\d} =N$ and use $x_{n_\d}$
as an approximate solution.
\end{enumerate}
}
\end{algorithm}

Let us give a brief remark on Algorithm \ref{alg3} when $N=1$. Note that, when both $\X$ and $\Y_0$ are Hilbert
spaces, $F_0:\X\to \Y$ is a bounded linear operator and $\Theta(x) =\frac{1}{2}\|x\|^2$, by taking
$J_s^{\Y_0} = J_2^{\Y_0}\equiv \mbox{id}$ and let $S_\varepsilon$ be the exact solver, i.e. $S_\varepsilon(\xi) = \xi$
for any $\xi \in \X$ and $\ep>0$, then (\ref{11.19.0}) becomes
\begin{equation}\label{11.20.3}
\left\{\begin{array}{lll}
\hat x_n = x_n + \frac{n}{n+\a} (x_n-x_{n-1}),\\[1.2ex]
x_{n+1}  = \hat x_n - \mu_n F_0^* \left(F_0 \hat x_n-y_0^\d\right).
\end{array}\right.
\end{equation}
This is exactly the formula (\ref{11.20.2}) applied to (\ref{11.20.1}) with $\varphi(x) = \frac{1}{2} \|F_0 x-y_0^\d\|^2$.
Therefore, it is reasonable to expect that Algorithm \ref{alg3} can converge faster than Algorithm \ref{alg1}.
Currently there is no available theory regarding the convergence of Algorithm \ref{alg3}, however, in the next section we
will use numerical examples to demonstrate its acceleration effect.

\subsection{Construction of inexact inner solvers}

In this subsection we will discuss how to find an inexact solver specified in Assumption \ref{A2}
for solving (\ref{min}) for each $\ep>0$. For those $\Theta$ such that the exact solution of (\ref{min})
can be determined explicitly, we can simply take each $S_\ep$ to be the exact solver. We therefore
consider only those $\Theta$ for which (\ref{min}) does not have an explicit exact solution. We will
focus on the total variation like functions which have significant importance in image reconstruction.

Due to the numerical implementation, we will give the exposition in a discrete setting. Let $\X = {\mathbb R}^{I\times J}$
and for each $z\in \X$ let $\|z\|_F$ denotes the Fr\"obenius norm of $z$. Let $D: \X \to \X\times \X$ be
the discrete gradient operator given by $D z = (D_1 z, D_2 z)$ for $z=(z_{i,j})\in \X$ with
\begin{eqnarray*}
(D_1 z)_{i,j} = \left\{ \begin{array}{lll}
z_{i+1,j}-z_{i,j}  & \mbox{if } i=1, \cdots, I-1; j=1, \cdots, J,\\
z_{1,j}-z_{I,j}  & \mbox{if } i=I; j = 1, \cdots, J
\end{array}\right.
\end{eqnarray*}
and
\begin{eqnarray*}
(D_2 z)_{i,j} = \left\{ \begin{array}{lll}
z_{i,j+1}-z_{i,j}  & \mbox{if } i=1, \cdots, I; j=1, \cdots, J-1,\\
z_{i,1}-z_{i,J}  & \mbox{if } i=1, \cdots, I; j = J.
\end{array}\right.
\end{eqnarray*}
We consider the function
\begin{equation}\label{TV}
\Theta(z) = \frac{1}{2\mu} \|z\|_F^2 + h(D z) + \iota_\C(z), \quad z\in \X,
\end{equation}
where $\mu>0$ is a constant, $\C\subset \X$ is a closed convex set representing the constraints on $z$,
$\iota_\C$ is the indicator function of $\C$, i.e.
\begin{equation*}
\iota_\C(z) = \left\{ \begin{array}{lll}
0 & \mbox{ if } z \in \C,\\
+\infty & \mbox{ otherwise},
\end{array}\right.
\end{equation*}
and $h$ is a function defined on $\X\times \X$ given by
\begin{equation*}
h(u, v) = \sum_{i=1}^I \sum_{j=1}^J \sqrt{u_{i,j}^2 +v_{i,j}^2}, \qquad (u, v) \in \X\times \X.
\end{equation*}
Note that $|z|_{TV}:= h(Dz)$ represents a discrete total variation of $z$. For this $\Theta$ the minimization
problem (\ref{min}) becomes
\begin{equation}\label{3.34}
x = \arg \min_{z\in \X} \left\{ \Psi_P(z):=\frac{1}{2\mu} \|z-\mu \xi\|_F^2 + h(D z) + \iota_\C(z)\right\},
\end{equation}
where $\Psi_P(z)$ is called a primal function. This is a total variation denoising problem \cite{ROF92} for 
which many algorithms have been developed
to solve it approximately. We will use the primal-dual hybrid gradient (PDHG) method introduced in \cite{ZC2008}
which is a special case of the Uzawa algorithm \cite{AHU58}. To formulate the method, we use
the Legendre-Fenchel conjugate $h^*$ of $h$ to rewrite (\ref{3.34}) as
\begin{equation*}
x = \arg \min_{z\in \X} \max_{\la \in \X\times \X}
\left\{ \Phi(z, \la):=\frac{1}{2\mu} \|z-\mu \xi\|_F^2 + \l \la, D z\r_F-h^*(\la) + \iota_\C(z)\right\}.
\end{equation*}
Then the PDHG method takes the form
\begin{align*}
\la_{k+1} & = \arg\max_{\la\in \X\times \X} \left\{{\rm \Phi}(z_k, \la)-\frac{1}{2\mu\tau_k} \|\la-\la_{k}\|_F^2\right\},\\
z_{k+1} & = \arg\min_{z\in \X} \left\{{\rm \Phi}(z, \la_{k+1}) + \frac{1-\theta_k}{2\mu\theta_k}\|z-z_k\|_F^2\right\}
\end{align*}
with suitably chosen step sizes $\tau_k>0$ and $0<\theta_k<1$. Direct manipulation shows that
\begin{align*}
\la_{k+1} & = \mbox{prox}_{\mu \tau_k h^*} \left(\la_k +\mu\tau_k D z_k\right), \\
z_{k+1} & = {\rm \Pi}_\C \left((1-\theta_k) z_k +\mu\theta_k \left(\xi - D^T \la_{k+1}\right)\right),
\end{align*}
where ${\rm \Pi}_\C$ denotes the projection onto $\C$ and $\mbox{prox}_{t h^*}$ for any $t>0$ denotes the proximal
mapping of $h^*$ defined by
\begin{equation*}
\mbox{prox}_{t h^*}(\bar \la) := \arg \min_{\la\in \X\times \X} \left\{h^*(\la) + \frac{1}{2 t} \|\la -\bar \la\|_F^2\right\},
\quad \bar \la \in \X\times \X.
\end{equation*}
For our $h$, it is easily seen that $h^*(\la) = \iota_{\mathcal Z} (\la)$, where
\begin{equation*}
{\mathcal Z}:= \left\{(u, v)\in \X \times \X: u_{i,j}^2 + v_{i,j}^2 \le 1 \mbox{ for } i=1, \cdots, I; j = 1, \cdots, J\right\}.
\end{equation*}
Thus $\mbox{prox}_{t h^*} (\la) = {\rm \Pi}_{\mathcal Z} (\la)$, where, for any $\la:=(u, v)\in \X\times \X$,
${\rm \Pi}_{\mathcal Z} (\la) = (y, z)$ with
\begin{equation*}
y_{i,j} = \frac{u_{i,j}}{\max\left\{1, \sqrt{u_{i,j}^2 +v_{i,j}^2}\right\}}, \quad
z_{i,j} = \frac{v_{i,j}}{\max\left\{1, \sqrt{u_{i,j}^2 +v_{i,j}^2}\right\}}.
\end{equation*}
To achieve a fast convergence, it was suggested in \cite{ZC2008} to choose the step sizes as
\begin{equation*}
\tau_k = 0.2+0.08k, \qquad \theta_k = \left(0.5- \frac{5}{15+k}\right)/\tau_k.
\end{equation*}
The convergence of the PDHG method, under such a choice of the step sizes, was confirmed in \cite{BR2012}.

In order to terminate the PDHG method, we use the relative duality gap. The dual function is given by
\begin{align*}
\Psi_D(\la) & = \min_{z\in \X} {\rm \Phi}(z, \la)\\
& = \left\{\begin{array}{lll}
\frac{1}{2\mu} \|(I-{\rm \Pi}_C)(\mu \xi -\mu D^T \la)\|_F^2+\frac{\mu}{2} \|\xi-D^T \la\|_F^2,  & \la \in {\mathcal Z},\\[0.8ex]
-\infty, & \la \not\in {\mathcal Z}.
\end{array}\right.
\end{align*}
Given a primal feasible point $z\in \X$ and a dual feasible point $\la\in \X\times \X$, we define the relative duality gap
\begin{equation*}
G_{rel}(z, \la) = \frac{\Psi_P(z)-\Psi_D(\la)}{|\Psi_P(z)|+|\Psi_D(\la)|}.
\end{equation*}
If there exists a feasible $(x, \la)$ such that $G_{rel}(x, \la)\le \eta$ for some $0<\eta<1$, then, by using the fact
that $\Psi_D(\la) \le \min_{z\in \X} \Psi_P(z)$, we can conclude that $x$ satisfies (\ref{min_inexact})
with
\begin{equation*}
\ep = \frac{2\eta}{1-\eta} \min_{z\in \X} \Psi_P(z).
\end{equation*}

\section{\bf Numerical simulations}

In this section we provide numerical simulations to test our theoretical result on Algorithm \ref{alg1} using
inexact inner solvers and to illustrate the acceleration effect embraced in Algorithm \ref{alg3}.
Our simulations were done by using MATLAB R2012a on a Lenovo laptop with Intel(R) Core(TM) i5 CPU 2.30GHz and 6GB
memory.

\begin{example}\label{ex1}
{\rm
We first consider the application of our algorithms in computed tomography which consists in determining
the density of cross sections of human body by measuring the attenuation of X-rays as they propagate
through the biological tissues \cite{N2001}. Mathematically, it requires to determine a function $f$
supported on a bounded domain from its Radon transform
\begin{align*}
{\mathscr R} f(\theta, \rho) &:= \int_{x\cos \theta+ y\sin \theta =\rho} f(x, y) d s \\
&=\int_{-\infty}^\infty f(\rho \cos \theta-t \sin \theta, \rho\sin\theta + t\cos \theta) dt,
\end{align*}
where $\theta \in [0, 2\pi)$ and $\rho \in {\mathbb R}$. The most prominent method in computed tomography
is the filtered backprojection (FBP) algorithm which is based on the explicit inversion formula and therefore
is fast and inexpensive. However, the FBP algorithm is not robust with respect to noise, its accuracy requires
patients to be over-exposed to X-rays, and it is difficult to incorporate a prior information into the algorithm.

In order to apply our algorithm to solve CT problems, we need a discrete model. We assume that the image
is supported on a square domain in ${\mathbb R}^2$ which is divided into $q \times q$ pixels numbered
from $1$ to $Q = q^2$. Each pixel $i$ is assigned a constant value $f_i$ such that the vector $f=(f_1, \cdots, f_Q)^T$
is a discrete version of the sought function. Assume that there are $M$ X-rays passing through the image.
Let $a_{ij}$ denote the length of the intersection of the $i$-th ray with the $j$-th pixel. Let $g_i$ be
the measurement of the attenuation for ray $i$ and let $g = (g_1, \cdots, g_M)^T$. Then
\begin{equation*}
g = A f
\end{equation*}
Note that $A$ is a sparse matrix of size $M \times Q$. We will apply our algorithms to solve this linear equation.

\begin{figure}[ht]
\centering
  \includegraphics[width = 0.48\textwidth]{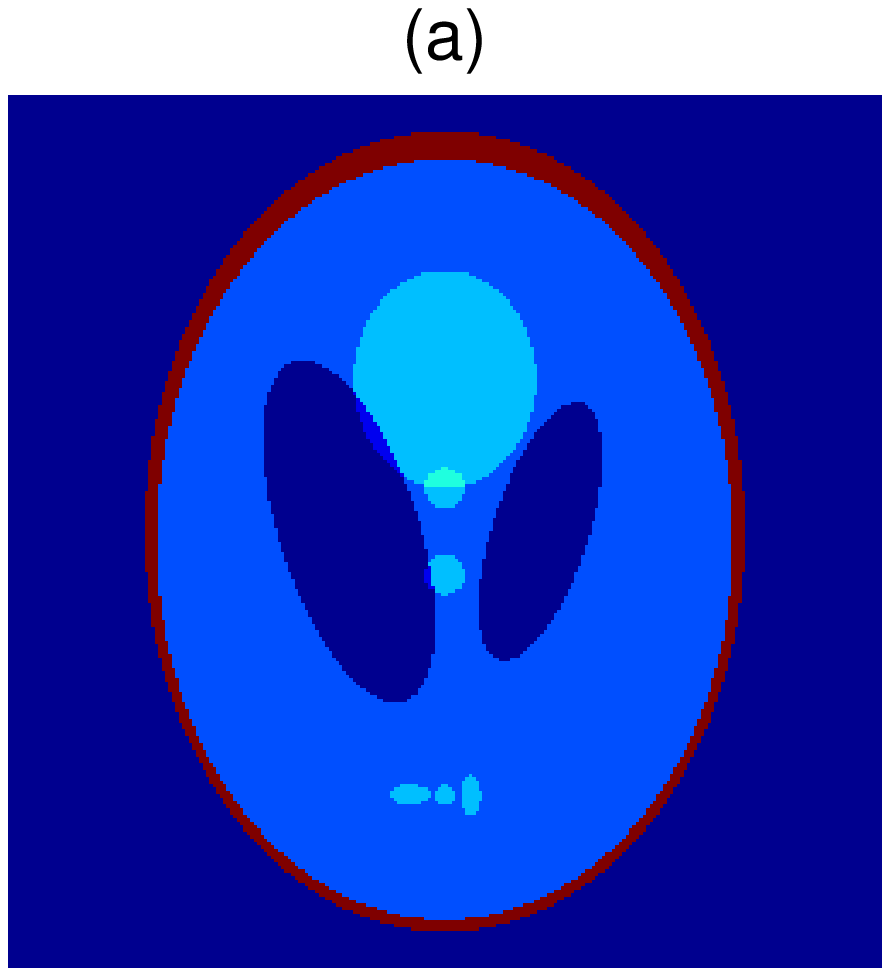}
  \includegraphics[width = 0.48\textwidth]{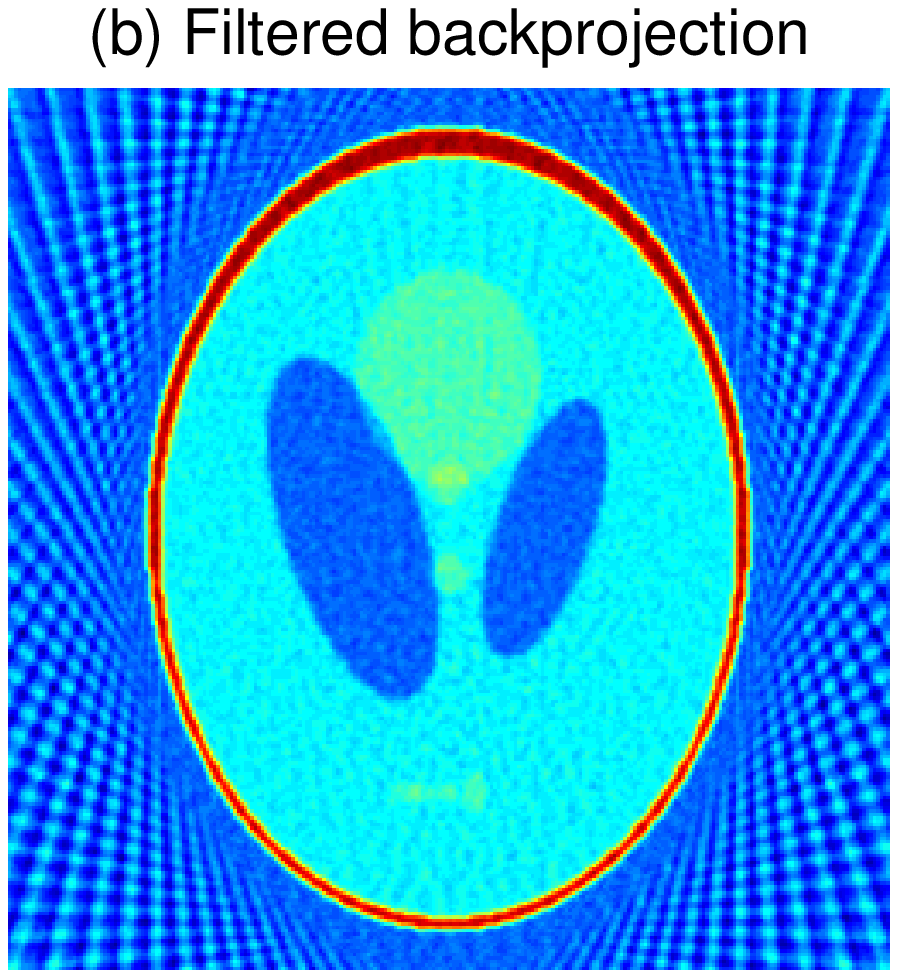}\\
  \includegraphics[width = 0.48\textwidth]{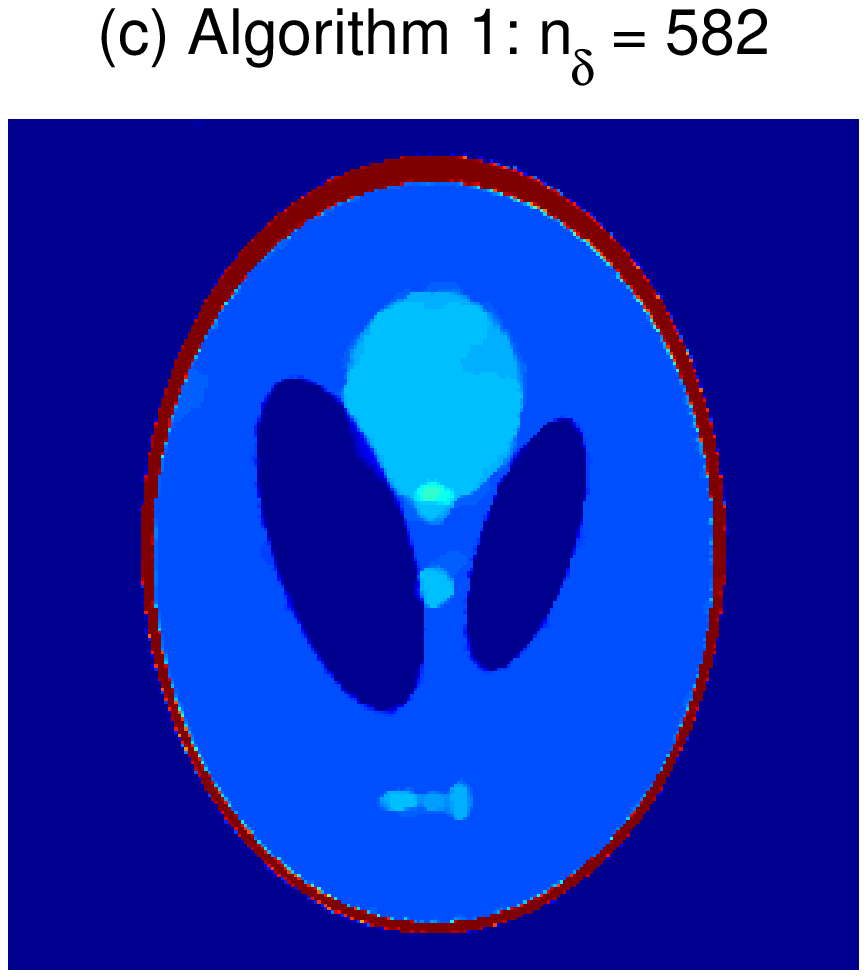}
   \includegraphics[width = 0.48\textwidth]{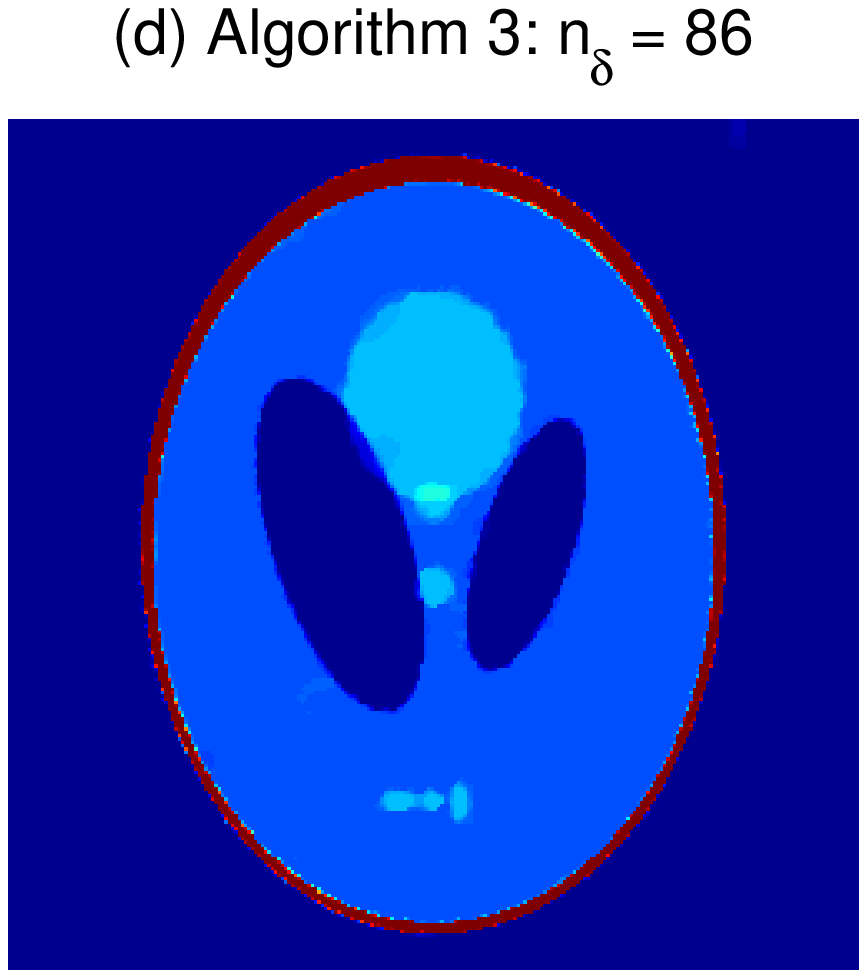}
  \caption{Reconstruction results by Algorithm \ref{alg1}, Algorithm \ref{alg3} and the filtered backprojection
  for the computed tomography in Example \ref{ex1}.}\label{fig1}
\end{figure}

The formation of the matrix $A$ depends on the scan geometry. In the following numerical simulations
we consider only test problems that model the standard 2D parallel-beam tomography; other scan geometries
can be done similarly. For the simulations, the true image $f_*$ is taken to be the Shepp-Logan phantom
shown in Figure \ref{fig1}(a) discretized on a $256\times 256$ pixel grid with its pixel values varying
in the interval $[0,1]$. This phantom is widely used in evaluating tomographic reconstruction algorithms.
We consider a full angle problem using 45 projection angles evenly distributed
between 1 and 180 degrees, with 367 lines per projection. The function \texttt{paralleltomo} in the MATLAB
package AIR TOOLS \cite{H2012} is used to generate the sparse matrix $A$, which has the dimension size
$M= 16515$ and $Q = 65536$.

\begin{figure}[ht]
\centering
  \includegraphics[width = 0.85\textwidth, height = 0.6\textwidth]{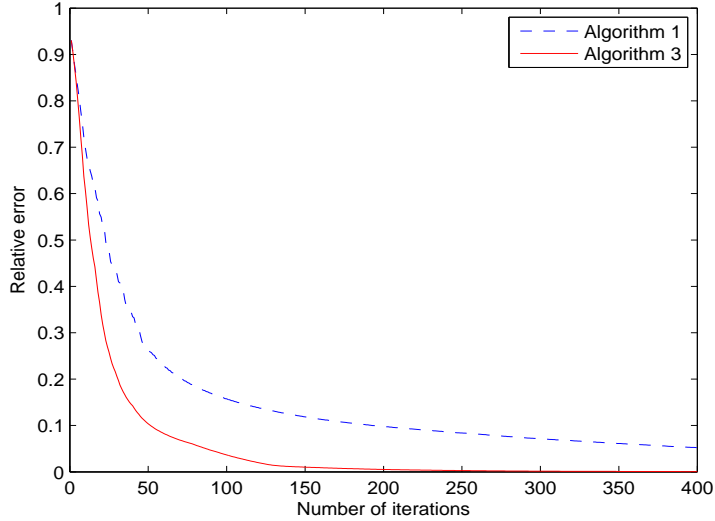}
  \caption{The plot of relative error of the solution versus iteration number by Algorithm \ref{alg1} and Algorithm \ref{alg3}
  for Example \ref{ex1}}\label{fig2}
\end{figure}

Let $g = A f_*$. We add Gaussian noise to $g$ to generate a noisy data $g^\delta$ with relative noise level
$\delta_{rel}:=\|g^\d-g\|_F/\|g\|_F = 0.01$ so that the noise level is $\d = \d_{rel} \|g\|_F$. We then
use $g^\d$ to reconstruct $f_*$ via Algorithm \ref{alg1} and Algorithm \ref{alg3} with $N=1$. We take
\begin{equation*}
\Theta(f) = \frac{1}{2\mu} \|f\|_F^2 + |f|_{TV} + \iota_\C(f)
\end{equation*}
with $\mu = 1$ and $\C= \{ f: f\ge 0\}$ and use the initial guess $f_0=\xi_0 =0$ and the parameters $\beta_0=
0.1/\mu$, $\beta_1 = 10$, $\sigma= 0.001$ and $\tau = 1.01$; we also take $\a = 5$ when using Algorithm \ref{alg3}.
The minimization problems associated with $\Theta$ are solved by the PDHG method which is terminated when
the relative duality gap is $\le (n+1)^{-2.2}$ at the $n$-th iteration.
In Figure \ref{fig1} (c) and (d) we present the reconstruction results by Algorithm \ref{alg1} and Algorithm \ref{alg3}
respectively. The both algorithms give satisfactory results. Algorithm \ref{alg1} terminates after 582 iterations
and takes $386$ seconds; while Algorithm \ref{alg3} terminates after 86 iterations and takes 26.6 seconds.
This clearly shows that Algorithm \ref{alg3} is much faster than Algorithm \ref{alg1}. As comparison, in
Figure \ref{fig1} (b) we include the reconstruction result by the FBP algorithm which is much faster but
the result is much worse.

To further illustrate the fast convergence property of Algorithm \ref{alg3}, we redo the numerical simulations
under the same situation but with exact data. We run $400$ iterations for both Algorithm \ref{alg1} and Algorithm
\ref{alg3}. The relation between the relative errors $\|f_n-f_*\|_F/\|f_*\|_F$ and the iteration numbers
is plotted in Figure \ref{fig2} which clearly shows that Algorithm \ref{alg3} has the acceleration effect.
}
\end{example}

\begin{example}\label{ex2}
{\rm We next consider the identification of the parameter $c$ in the boundary value problem
\begin{eqnarray}\label{PDE}
\left\{
  \begin{array}{ll}
    -\triangle u + cu = f \qquad  \mbox{in } \Omega, \\
    u = g  \qquad \mbox{on } \partial\Omega
  \end{array}
\right.
\end{eqnarray}
from an $L^2(\Omega)$-measurement of the state $u$, where $\Omega\subset\mathbb{R}^2$ is a bounded domain with Lipschitz boundary,
$f\in L^2(\Omega)$ and $g\in H^{3/2}(\Omega)$. This is a benchmark example of nonlinear inverse problems. We assume that the exact
solution $c^{\dag}$ is in $L^2(\Omega)$. This problem reduces
to solving $F(c) =u$, if we define the nonlinear operator $F: L^2(\Omega)\rightarrow  L^2(\Omega)$ by $F(c): = u(c)$,
where $u(c)\in H^2(\Omega)\subset  L^2(\Omega)$ is the unique solution of (\ref{PDE}). This operator $F$
is well defined on
\begin{equation*}
\D(F): = \left\{ c\in  L^2(\Omega) : \|c-\hat{c}\|_{ L^2(\Omega)}\leq \gamma_0 \mbox{ for some }
\hat{c}\geq0,\ \textrm{a.e.}\right\}
\end{equation*}
for some positive constant $\gamma_0>0$. It is known that $F$ is Fr\'{e}chet differentiable; the Fr\'{e}chet derivative
of $F$ and its adjoint are given by
\begin{equation}\label{Frechet}
F'(c)h  = -A(c)^{-1}(hF(c)) \quad \mbox{and} \quad F'(c)^* w  = -u(c) A(c)^{-1}w
\end{equation}
for $h, w\in L^2(\Omega)$, where $A(c): H^2(\Omega)\cap H_0^1(\Omega)\rightarrow  L^2(\Omega)$ is
defined by $A(c) u = -\triangle u +cu$ which is an isomorphism uniformly in ball $B_{2\rho}(c^{\dag})$
for small $\rho>0$. Moreover, Assumption \ref{A1} (c) holds for small $\rho>0$ (see \cite{EHN96}).

\begin{figure}[ht]
\centering
\includegraphics[width = 0.49\textwidth]{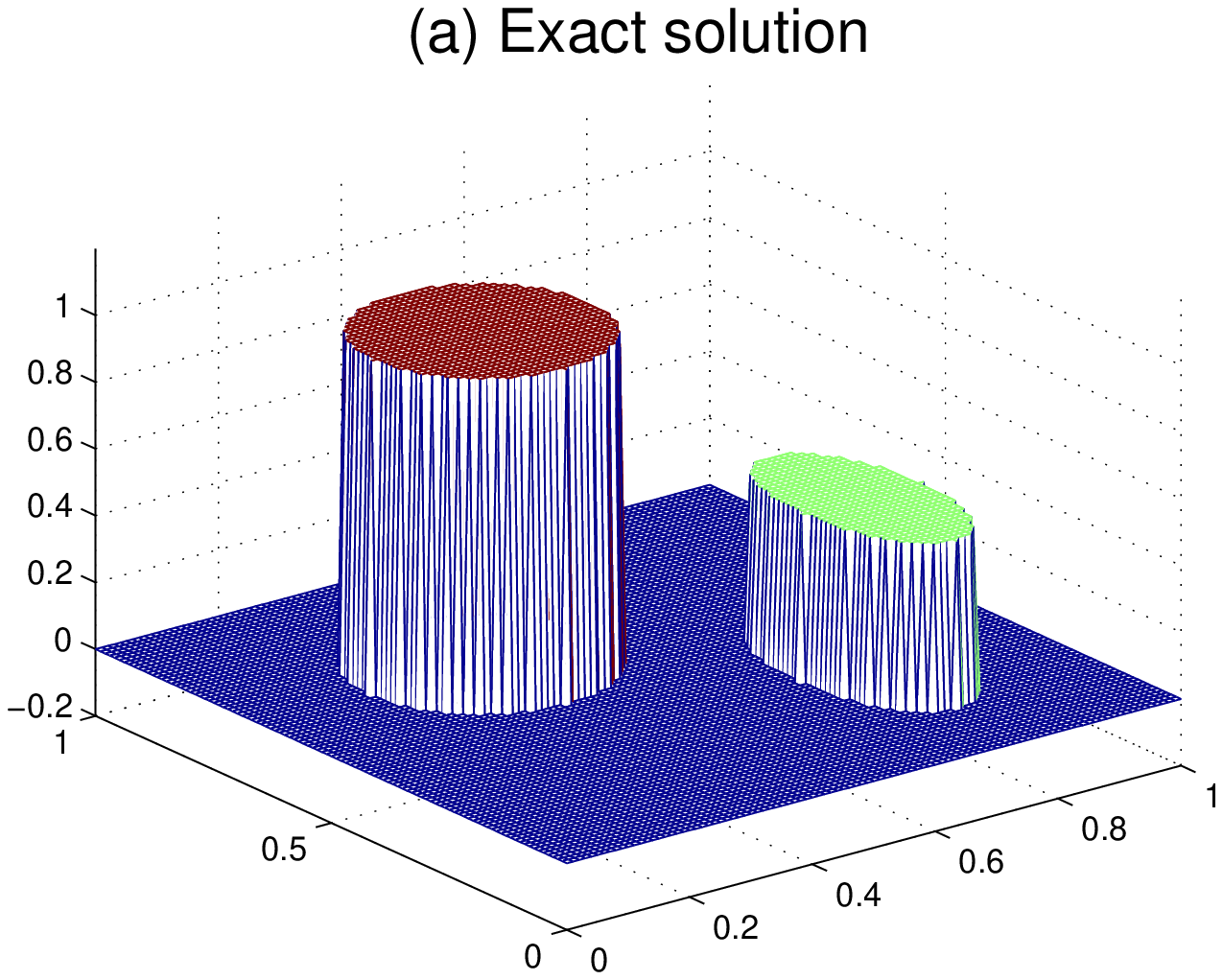}\\
  \includegraphics[width = 0.49\textwidth]{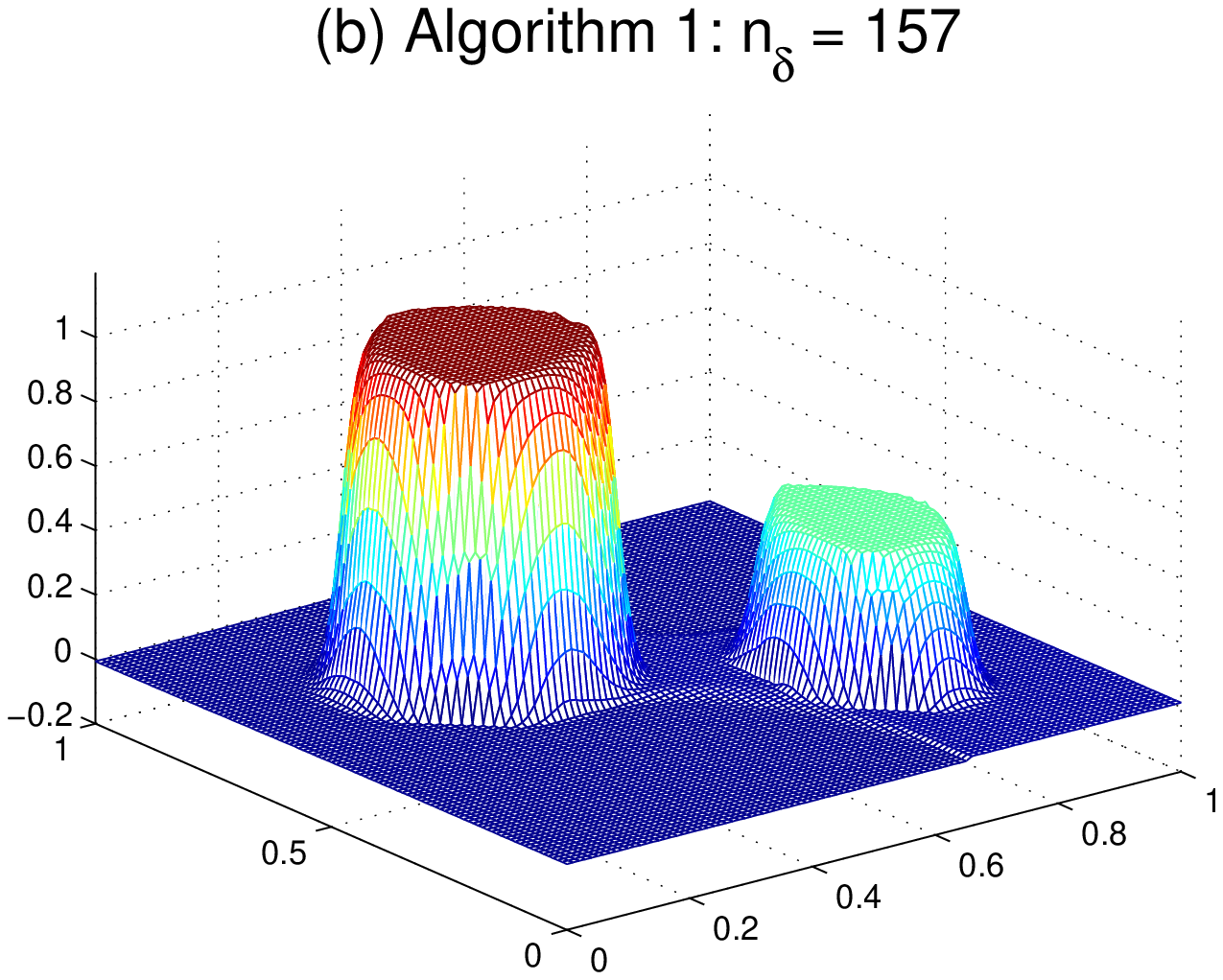}
   \includegraphics[width = 0.49\textwidth]{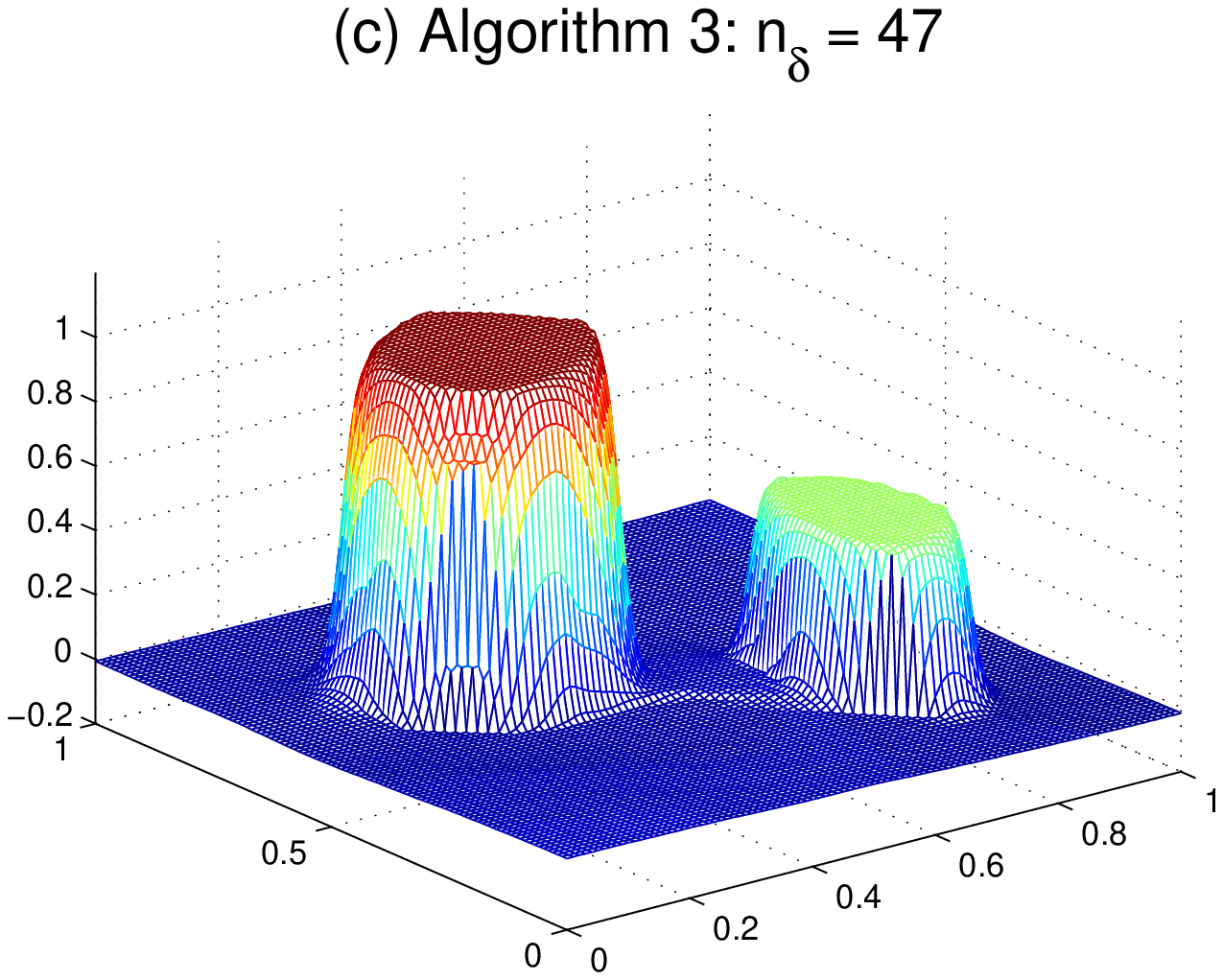}
  \caption{Reconstruction results by Algorithm \ref{alg1} and Algorithm \ref{alg3} for the parameter identification
    problem in Example \ref{ex2}.}\label{fig3}
\end{figure}

We will present our numerical results on $\Omega=[0,1]\times [0,1]$ with $g \equiv 1$ on $\p \Omega$
and
\begin{equation*}
f(x, y) = 200 e^{-10(x-0.5)^2 -10 (y-0.5)^2} \quad \mbox{ in } \Omega.
\end{equation*}
We assume the sough solution $c_*$ is a piecewise constant function as shown in Figure \ref{fig3} (a)
and reconstruct it by using a noisy data $u^\d$ with relative noise level $\d_{rel}:=\|u^\d-u\|_{L^2}/\|u\|_{L^2}
= 0.46\times 10^{-3}$. When applying Algorithm \ref{alg1} and Algorithm \ref{alg3} with $N=1$, we take
$\X = \Y= L^2(\Omega)$ and 
\begin{equation*}
\Theta(c) = \frac{1}{2\mu} \|c\|_{L^2(\Omega)}^2 + \int_\Omega |Dc|
\end{equation*}
with $\mu = 20$, where 
$$
\int_\Omega |Dc| = \sup\left\{\int_\Omega c \, \mbox{div} \varphi dx: \varphi\in C_0^1(\Omega, {\mathbb R}^2)
\mbox{ and } \|\varphi\|_{L^\infty} \le 1 \right\}
$$
denotes the total variation of $c$. We use the initial guess $c_0=\xi_0 =0$ and the parameters $\beta_0 =0.01/\mu$,
$\beta_1 =2\times 10^4$, $\sigma = 0.001$ and $\tau = 1.02$; we also take $\a=5$ when using Algorithm \ref{alg3}.
In order to carry out the computation, we need to discretize the problem. We divide $\Omega$ into $100\times 100$ 
small squares of equal size and solve all partial differential equations involved approximately by a finite difference method.
We also discretize $\Theta(c)$ so that $\|c\|_{L^2}$ is replaced by the Fr\"{o}benius norm of arrays
and $\int_\Omega |D c|$ is replaced by the discrete total variation given in Section 3.5. 
The corresponding minimization problems associated with the discrete $\Theta$ are solved by the PDHG method 
which is terminated as long as the relative duality gap is $\le (n+1)^{-1.5}$ at the $n$-th iteration.
In Figure \ref{fig3} (b) and (c) we report the computational results using Algorithm \ref{alg1} and
Algorithm \ref{alg3}. The both algorithms give satisfactory reconstruction results. Moreover,
Algorithm \ref{alg3} requires significantly less number of iterations than Algorithm \ref{alg1} which
demonstrates that Algorithm \ref{alg3} has the acceleration effect.

To further illustrate the fast convergence property of Algorithm \ref{alg3}, we redo the above
computation using exact data. We perform 100 iterations for both Algorithm \ref{alg1} and Algorithm \ref{alg3}.
The curve of the relative error $\|c_n-c_*\|_{L^2}/\|c_*\|_{L^2}$ versus the iteration number $n$
is plot in Figure \ref{fig4} which clearly indicates that Algorithm \ref{alg3} converges faster than
Algorithm \ref{alg1}.

\begin{figure}[ht!]
     \begin{center}
        {
           \includegraphics[width = 0.9\textwidth, height = 0.6\textwidth]{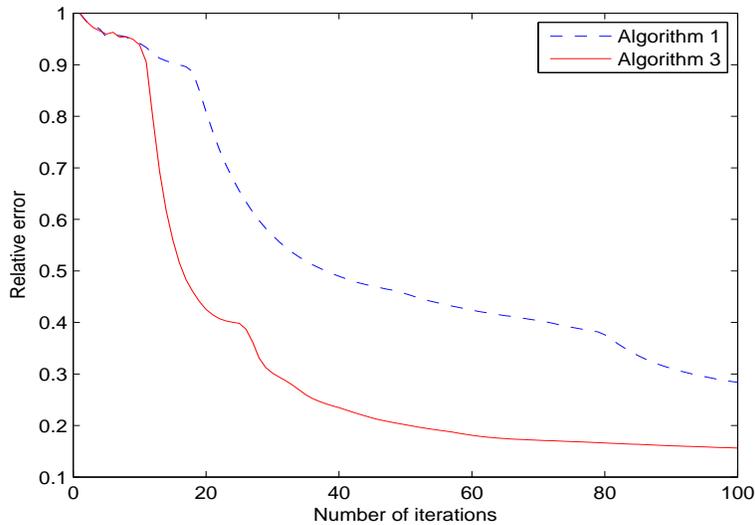}
        }
    \end{center}
    \caption{The plot of relative error of the solution versus iteration number by Algorithm \ref{alg1} and Algorithm \ref{alg3}
    for Example \ref{ex2}.}%
   \label{fig4}
\end{figure}
}
\end{example}

\section*{\bf Acknowledgement}

This work is partially supported by the Discovery Project grant of Australian Research Council.

\end{document}